\documentclass[psamsfonts]{amsart}
\usepackage{amssymb}
\usepackage{graphicx,color}
\usepackage{amssymb,amsfonts,amsthm}
\usepackage[all,arc]{xy}
\usepackage{enumerate}
\usepackage{mathrsfs}

\newtheorem{thm}{Theorem}[section]

\newtheorem{lemma}[thm]{Lemma}
\newtheorem{prop}[thm]{Proposition}
\newtheorem{cor}[thm]{Corollary}

\theoremstyle{definition}
\newtheorem{defn}[thm]{Definition}

\theoremstyle{remark}
\newtheorem{rem}[thm]{Remark}
\makeatletter
\let\c@equation\c@thm
\makeatother
\numberwithin{equation}{section}

\title[Fully Nonlinear Loewner-Nirenberg Problem]{On Fully Nonlinear Loewner-Nirenberg Problem of Ricci curvature}

\thanks{The author is supported by
National Natural Science Foundation of China (No. 12001138).}

\author{Zhenan Sui}

\address{Institute for Advanced Study in Mathematics of HIT, Harbin Institute of Technology, Harbin, China}
\email{sui.4@osu.edu}

\begin{document}

\begin{abstract}
We prove the existence of a smooth complete conformal metric with prescribed kth elementary symmetric function of negative Ricci curvature under certain condition on general domain in Euclidean space. We then formulate this problem for more general equations.
\end{abstract}

\subjclass[2010]{Primary 53C21; Secondary 35J60}

\maketitle


\section {\large Introduction}

\vspace{4mm}

In this paper, we discuss the fully nonlinear version of Loewner-Nirenberg problem: let $\Omega \subsetneq \mathbb{R}^n$ be a domain in Euclidean space with $n \geq 3$ and $g$ the Euclidean metric. Assume that $\partial \Omega$ consists of a finite number of disjoint, non-self-intersecting, smooth, compact submanifolds of $\mathbb{R}^n$. We want to find a smooth complete metric
$g_u = u^{\frac{4}{n - 2}} g$ with $u > 0$, which satisfies
\begin{equation} \label{eq1-1}
\sigma_k \Big( - g_u^{-1} \text{Ric}_{g_u} \Big) := \sigma_k \Big( \lambda \big( - g_u^{-1} \text{Ric}_{g_u} \big) \Big) = 1 \quad \text{  in  } \Omega,
\end{equation}
where $\text{Ric}_g$ is the Ricci tensor of $g$, $\lambda( A ) = ( \lambda_1, \ldots,
\lambda_n )$ are the eigenvalues of the matrix $A$, and
\[ \sigma_k (\lambda) =  \sum\limits_{ 1 \leq i_1 <
\cdots < i_k \leq n} \lambda_{i_1} \cdots \lambda_{i_k} \]
is the $k$th elementary symmetric function defined on
Garding's cone
\[\Gamma_k = \{ \lambda  \in \mathbb{R}^n : \sigma_j ( \lambda ) > 0, \, j = 1,
\ldots, k \}. \]

Under conformal deformation of metric $g_u = u^{\frac{4}{n - 2}} g$, the Ricci tensor $\text{Ric}_{g_u}$ and $\text{Ric}_{g}$ are related by the formula
\begin{equation*}
\frac{n - 2}{2} \mbox{Ric}_{g_u} = \frac{n - 2}{2} \mbox{Ric}_g - (n - 2) \frac{\nabla^2 u}{u} - \bigg( \frac{\Delta u}{u} + \frac{|\nabla u|^2}{u^2} \bigg) g + \frac{n}{u^2} du
\otimes du,
\end{equation*}
where $\nabla{u}$, $\nabla^{2}u$ and $\Delta u$ are the
gradient, Hessian and Laplace-Beltrami operator of $u$ with respect to the background
metric $g$ respectively. We note that $\mbox{Ric}_g \equiv 0$ if $g$ is Euclidean.

Consequently, the geometric problem \eqref{eq1-1} is equivalent to the second order partial differential equation
\begin{equation} \label{conformal2}
\sigma_k^{\frac{1}{k}} \big( W [ u ] \big)  =  \frac{n - 2}{2} u^{\frac{n+2}{n - 2}} \quad \text{ in } \Omega
\end{equation}
along with the boundary condition
\begin{equation} \label{eq1-2}
u =  \infty \quad \text{ on }  \partial \Omega,
\end{equation}
where
\begin{equation} \label{wu}
W [ u ] = g^{-1} \bigg( (n - 2) \nabla^2 u - \frac{n}{u}  du \otimes du + \Big( \Delta u + \frac{|\nabla u|^2}{u} \Big) g - \frac{n - 2}{2} u \text{Ric}_g \bigg).
\end{equation}
We shall call a  $C^2$ function $u$
$k$-admissible (or admissible if there is no ambiguity) in $\Omega$ if
\begin{equation*} \label{admissi}
\lambda \big( W[u] \big)(x) \in \Gamma_k
\end{equation*}
for any $x \in \Omega$. We note that equation \eqref{conformal2} is elliptic if $u$ is admissible.
For convenience, we call equation \eqref{conformal2} the conformal $k$-Ricci curvature equation.
Our goal in this paper is to seek smooth positive admissible solution to \eqref{conformal2}--\eqref{eq1-2}.

When $k = 1$, \eqref{conformal2}--\eqref{eq1-2} reduces to
\begin{equation} \label{scalar_eqn}
\left\{ \begin{aligned}
\frac{4 (n - 1)}{n - 2} \Delta u - S_g  u = & u^{\frac{n + 2}{n - 2}} \quad \text{ in } \,\, \Omega, \\
u = & \infty \quad \text{ on } \,\, \partial \Omega,
\end{aligned} \right.
\end{equation}
where $S_g$ is the scalar curvature with respect to $g$. In this case, the above problem is known as Loewner-Nirenberg problem. If $\Gamma$ is a smooth compact submanifold of $\mathbb{R}^n$ of codimension $m$ such that $\partial\Omega \setminus \Gamma$ is also compact, Loewner and Nirenberg \cite{Loewner} proved the existence of a smooth conformally flat metric with constant negative scalar curvature in $\Omega$ satisfying \[ u(x) \rightarrow \infty \quad \text{ as } x \rightarrow \Gamma \]
if $m < \frac{n}{2} + 1$. They also gave the nonexistence result if $m > \frac{n}{2} + 1$ and conjectured the nonexistence for the borderline case $m = \frac{n}{2} + 1$, which was later proved by Aviles and V\'eron \cite{Aviles, Veron}. Other related work can be found in  \cite{AILA18, ACF92, AM88, Finn98, GW17, Graham17, HJS20, HS20, Jiang21, Mazzeo91} and the references therein.

When $k > 1$, equation \eqref{conformal2} becomes fully nonlinear. Following the literature, we call the above problem fully nonlinear Loewner-Nirenberg problem. This problem and its analogues have been studied by many researchers, see for instance \cite{CHY05, Gonzalez-Li-Nguyen, Guan08, Gursky-Streets-Warren, GV03, LN21, LNX23, LS05, Sui17, Wang21} and the references therein. If $\text{Ric}_g$ in \eqref{eq1-1} is replaced by the Schouten tensor $A_g$:
\[ A_g = \frac{1}{n - 2} \Big( Ric_g - \frac{S_g}{2 (n - 1)}  g \Big), \]
Gonz\'alez, Li and Nguyen \cite{Gonzalez-Li-Nguyen} proved the existence of a Lipschitz continuous complete conformally flat metric $g_u = u^{\frac{4}{n - 2}} g$ satisfying
\begin{equation*}
\sigma_k \Big(  - g_u^{-1} A_{g_u}  \Big) = 1
\end{equation*}
if the vector
\[ u_m = \big( \underbrace{1, \ldots, 1}_{n - m + 1}, \underbrace{- 1, \ldots, - 1}_{m - 1} \big) \in \Gamma_k. \]
They also gave the nonexistence result if $u_m \notin \overline{\Gamma}_k$. The borderline case
\begin{equation} \label{borderline}
u_m \in \partial\Gamma_k
\end{equation}
remains open.

Motivated by the above literatures, we investigate fully nonlinear Loewner-Nirenberg problem of Ricci curvature. Different from the Schouten tensor case, we are able to obtain smooth solutions, thanks to the  second order estimates of Guan \cite{Guan08} and Evans--Krylov theory \cite{Evans, Krylov}.

Our first result is an existence result when $\partial\Omega$ is composed of hypersurfaces.

\begin{thm} \label{theorem2}
Let $\Omega \subsetneq \mathbb{R}^n$ be a domain whose boundary $\partial\Omega$ consists of finitely many disjoint non-self-intersecting smooth compact hypersurfaces. Then there exists a smooth complete metric $g_u = u^{\frac{4}{n - 2}} g$ satisfying \eqref{eq1-1}, or equivalently, there exists a smooth positive admissible solution $u$ to \eqref{conformal2}--\eqref{eq1-2}. Moreover, the conformal factor $u$ has the following growth rate
\begin{equation}  \label{growthrate}
{\rho}^{ \frac{n}{2} - 1}(x) u(x) \rightarrow  \big( (n - 1) (C_n^k)^{\frac{1}{k}} \big)^{\frac{n - 2}{4}} \quad \text{ as }  x \rightarrow \partial\Omega,
\end{equation}
where $\rho (x)$ is the distance of $x$ to $\partial \Omega$.
If $\Omega$ is bounded or $k = 1$, the solution $u$ is unique.
\end{thm}
We remark that on smooth compact Riemannian manifold with boundary, the existence result was given in Guan \cite{Guan08} and Gursky, Streets and Warren \cite{Gursky-Streets-Warren}; in addition, \cite{Gursky-Streets-Warren} derived the growth rate of solution near the boundary. In Theorem \ref{asymptotic_near_boundary}, we shall adopt the idea in \cite{Gursky-Streets-Warren} but use a simplified subsolution and supersolution to derive the growth rate \eqref{growthrate} on general smooth manifolds with compact boundaries.

Our main result on the solvability of fully nonlinear Loewner-Nirenberg problem of Ricci curvature is as follows.

\begin{thm} \label{theorem1}
Let $\Omega \subsetneq \mathbb{R}^n$ be a domain. Suppose that the boundary $\partial \Omega$ consists of finitely many disjoint non-self-intersecting smooth compact submanifolds, and the maximal solution $u_{\Omega}$ of \eqref{conformal2} satisfies $u_{\Omega} \not\equiv 0$. If each component of $\partial \Omega$ has codimension $m$ satisfying
\begin{equation} \label{nsconditionintroduction}
 v_m = \big( \underbrace{n - m, \ldots, n - m}_{n - m + 1}, \underbrace{2 - m, \ldots, 2 - m}_{m - 1} \big) \in \Gamma_k,
\end{equation}
then there exists a smooth complete metric $g_u = u^{\frac{4}{n - 2}} g$ satisfying \eqref{eq1-1}, or equivalently, there exists a smooth positive admissible solution to \eqref{conformal2}--\eqref{eq1-2}. If some component has codimension $m$ satisfying $v_m \notin \overline{\Gamma}_k$, then there does not exist a complete conformal metric satisfying \eqref{eq1-1}, or equivalently, equation \eqref{conformal2}--\eqref{eq1-2} is not solvable.
\end{thm}
The proof draws on ideas from \cite{Loewner, Gonzalez-Li-Nguyen}. The borderline case $v_m \in \partial\Gamma_k$ is still open. The difficulty lies in the lack of a suitable upper barrier. We observe that when $k = 1$, the borderline case \eqref{borderline} agrees with \eqref{nsconditionintroduction}, which was solved in \cite{Aviles, Veron} by some integral estimates. We wish to reinvestigate the borderline case in the future.

Theorem \ref{theorem2} and Theorem \ref{theorem1} concern the existence and non-existence of smooth solution to \eqref{conformal2}--\eqref{eq1-2}. It should be noted that, if one merely considers continuous (or Lipschitz continuous) solution, the result follows from previous result of Gonz\'alez, Li and Nguyen \cite{Gonzalez-Li-Nguyen}, since \eqref{conformal2} can be recast as an equation for the Schouten tensor.
For the Schouten curvature case, the result in \cite{Gonzalez-Li-Nguyen} is optimal for general domains as described above, in view of the counterexamples to $C^1$ regularity given in Li and Nguyen \cite{LN21} and in Li, Nguyen and Xiong \cite{LNX23}. Because conformal Ricci curvature equations such as equation \eqref{conformal2} contain the $\Delta u$ term, it is possible to obtain a priori estimates up to second order. Combined with Evans-Krylov theory, smooth solution can be obtained. For further discussions concerning the existence and regularity of solutions on general Riemannian manifolds and for a large class of defining functions of the equation, one may see the recent work by Yuan \cite{Yuan22} and Duncan and Nguyen \cite{DN}.

Motivated by Guan and Zhang \cite{Guan-Zhang}, the second goal in this paper is to investigate the existence of complete conformal metric subject to the following more general equation
\begin{equation} \label{eq4-0}
\sigma_k \Big( - g_u^{-1} \text{Ric}_{g_u} \Big) + \alpha(x) \sigma_{k - 1} \Big( - g_u^{-1} \text{Ric}_{g_u} \Big) = \alpha_0(x),
\end{equation}
where $\alpha_0 (x) > 0$ and $\alpha(x)$ are real valued smooth functions.
Under conformal deformation of metric $g_u = u^{\frac{4}{n - 2}} g$, \eqref{eq4-0} is equivalent to
\begin{equation} \label{eq4-1}
\begin{aligned}
\sigma_k \bigg( \frac{2}{n - 2} u^{- \frac{n + 2}{n - 2}}  W[ u ] \bigg) + \alpha(x) \sigma_{k - 1} \bigg( \frac{2}{n - 2} u^{- \frac{n + 2}{n - 2}} W[ u ] \bigg) = \alpha_0(x).
\end{aligned}
\end{equation}
As proved in \cite{Guan-Zhang}, equation \eqref{eq4-1} is elliptic when $u$ is $(k - 1)$-admissible.
In what follows, we may assume that $k \geq 2$ in \eqref{eq4-1}.

Our first result associated to \eqref{eq4-1} shows that similar to \eqref{conformal2}, there exists smooth positive $(k - 1)$-admissible solution to Dirichlet problem.

\begin{thm} \label{Theorem1-1}
Let $\Omega$ be a bounded domain in $\mathbb{R}^n$ whose boundary is composed of finitely many disjoint non-self-intersecting smooth compact hypersurfaces, and $\alpha_0 (x) > 0$, $\alpha(x)$ be real valued smooth functions on $\overline{\Omega}$. For any smooth positive function $\varphi$ defined on $\partial\Omega$, there exists a smooth positive $(k - 1)$-admissible solution $u$ to equation \eqref{eq4-1} which satisfies the boundary condition
\begin{equation} \label{eq1-3}
u = \varphi \quad \text{  on  } \partial\Omega.
\end{equation}
\end{thm}
The proof of Theorem \ref{Theorem1-1} relies on the establishment of $C^2$ a priori estimates (see Theorem \ref{global gradient estimate}, \ref{global second order estimate}, \ref{second order boundary estimate}) and the existence of a subsolution. We notice that the method for $C^2$ estimates is similar to Guan \cite{Guan08} but may be more complicated due to some extra terms. In addition, the estimates in section 4 depend explicitly on $\inf \alpha_0$. Also, we give a new proof of the second order boundary estimate. For deriving preliminary estimates and conducting continuity process, we construct a smooth subsolution to the Dirichlet problem \eqref{eq4-1}--\eqref{eq1-3}, synthesizing the ideas of Guan \cite{Guan08} and Guan \cite{GuanP02}. It is worth mentioning that on smooth closed manifolds, equations in the form \eqref{eq4-0} and its Schouten analogue have been studied by Chen, Guo and He \cite{CGH22}, where second order estimates and existence results are established.

Our next results associated to \eqref{eq4-1} are on the formulation of fully nonlinear Loewner-Nirenberg type problem. When the boundary $\partial\Omega$ is composed of hypersurfaces, we have the following result.
\begin{thm} \label{Theorem6-2}
Let $\Omega$ be a domain in $\mathbb{R}^n$ with nonempty boundary which is composed of finitely many disjoint non-self-intersecting smooth compact hypersurfaces,
and $\alpha_0 (x) > 0$, $\alpha(x)$ be real valued smooth functions on $\overline{\Omega}$.
Then there exists a smooth complete metric $g_u = u^{\frac{4}{n - 2}} g$ satisfying \eqref{eq4-0}, or equivalently, there exists a smooth positive $(k - 1)$-admissible solution $u$ to equation \eqref{eq4-1}--\eqref{eq1-2}, provided that
\[  \inf\limits_{\Omega} \alpha_0(x) > 0, \quad   \inf\limits_{\Omega} \alpha(x) > - \infty, \quad  \sup\limits_{\Omega} \alpha_0(x) < \infty, \quad  \sup\limits_{\Omega} \alpha(x) < \infty.  \]
\end{thm}
The proof of Theorem \ref{Theorem6-2} relies further on the interior second order estimates for equation \eqref{eq4-1}, which are derived in Theorem \ref{interior gradient estimate}, \ref{second order interior estimate}. Also, we need to find a global upper barrier and a global lower barrier to conduct the diagonal process. The latter is given by Proposition \ref{Proposition6-1} and the former is provided by Loewner and Nirenberg \cite{Loewner}.

When $\Omega$ is a general domain in $\mathbb{R}^n$ with smooth compact boundary, in order to define the maximal solution (see Definition \ref{Maximal solution}), we have to assume that $\alpha(x) \leq 0$ in $\Omega$ so that we can apply the maximum principle Theorem \ref{Theorem6-1}. We obtain the following result.

\begin{thm} \label{Theorem1-2}
Let $\Omega \subsetneq \mathbb{R}^n$ be a domain with finitely many disjoint non-self-intersecting smooth compact submanifolds as the boundary. Suppose that $\alpha_0 (x) > 0$ and $\alpha(x)$ are real valued smooth functions in $\Omega$ satisfying
\[\alpha(x) \leq 0  \text{  in  } \Omega, \quad   \inf\limits_{\Omega} \alpha_0(x) > 0, \quad   \inf\limits_{\Omega} \alpha(x) > - \infty, \quad \sup\limits_{\Omega} \alpha_0(x) < \infty.   \]
In addition, suppose that the maximal solution $u_{\Omega}$ of \eqref{eq4-1} satisfies
\[ u_{\Omega} \not\equiv 0 \quad \text{  in  } \Omega. \]
If each component of $\partial \Omega$ has codimension $m$ satisfying $v_m \in \Gamma_k$,
then there exists a smooth complete metric $g_u = u^{\frac{4}{n - 2}} g$ satisfying \eqref{eq4-0}, or equivalently, there exists a smooth positive $k$-admissible solution to \eqref{eq4-1}--\eqref{eq1-2}. If some component has codimension $m$ satisfying $v_m \notin \overline{\Gamma}_k$, then there does not exist a complete conformal metric satisfying \eqref{eq4-0}, or equivalently, equation \eqref{eq4-1}--\eqref{eq1-2} is not solvable.
\end{thm}

This paper is organized as follows. We prove Theorem \ref{theorem2} in section 2 and Theorem \ref{theorem1} in section 3. Section 4 is devoted to $C^2$ a priori estimates and interior $C^2$ estimates for equation \eqref{eq4-1}. Then we apply continuity method and degree theory to prove Theorem \ref{Theorem1-1}. In section 5, we discuss the fully nonlinear Loewner-Nirenberg problem associated to equation \eqref{eq4-1} and give the proof of Theorem \ref{Theorem6-2} and Theorem \ref{Theorem1-2}.

\medskip
\noindent
{\bf Acknowledgements} \quad
The author would like to express the deep thanks to the reviewer for the valuable suggestions and references. The author is supported by National Natural Science Foundation of China (No. 12001138).

\vspace{4mm}

\section{Solutions on Euclidean domains with smooth compact boundary consisting of closed hypersurfaces}

\vspace{4mm}

In this paper, we shall mainly discuss admissible solutions to equation \eqref{conformal2}-\eqref{eq1-2} in Euclidean domains which have smooth compact boundaries. Within this section, we focus on the case when these boundaries are composed of closed hypersurfaces.

\vspace{2mm}

\subsection{Preliminaries}~

\vspace{2mm}

\begin{defn}
A function $0 < \underline{u} \in C^2(\Omega)$ is a subsolution of \eqref{conformal2} in $\Omega$ if
\[  \lambda\big( W [\underline{u}] \big) \in \Gamma_k  \text{    and    } \sigma_k^{\frac{1}{k}} \big( W [ \underline{u} ] \big)  \geq  \frac{n - 2}{2} \underline{u}^{\frac{n + 2}{n - 2}} \quad \text{  in  }  \Omega.     \]
A function $0 < \overline{u} \in C^2(\Omega)$ is a supersolution of \eqref{conformal2} in $\Omega$ if
\[ \text{either  } \lambda\big( W [\overline{u}] \big) \notin \Gamma_k  \text{  or  }
\sigma_k^{\frac{1}{k}} \big( W [ \overline{u} ] \big)  \leq  \frac{n - 2}{2} \overline{u}^{\frac{n + 2}{n - 2}} \quad \text{  in  }  \Omega.     \]
\end{defn}

In view of the form of equation \eqref{conformal2}, we first give a general property on supersolutions and subsolutions.
\begin{prop} \label{prop}
If $u$ is a positive subsolution of \eqref{conformal2}, so does $c u$ for any constant $0 < c < 1$.  If $u$ is a positive supersolution of \eqref{conformal2}, so does $c u$ for any constant $c > 1$; if in addition $\lambda\big( W [u] \big) \notin \Gamma_k$ everywhere in $\Omega$, then $c u$ is a supersolution of \eqref{conformal2} for any $c > 0$.
\end{prop}
\begin{proof}
The conclusion follows from the fact that
$W[c u] = c  W[u]$ and
\[ \sigma_k^{\frac{1}{k}} ( W[ c u ] ) = \sigma_k^{\frac{1}{k}} (c W[u]) = c \sigma_k^{\frac{1}{k}} (W[u]) \quad \text{if  } \lambda\big( W[ u ] \big) \in \Gamma_k. \]
\end{proof}

We will also need the following maximum principle. The proof is similar to \cite{Gursky-Streets-Warren}.
\begin{thm} \label{MP}
Let $M$ be a smooth compact manifold with boundary. Suppose that $u$ and $v$ are $C^2$ positive subsolution and supersolution of \eqref{conformal2} respectively on $M$. If $u \leq v$ on $\partial M$, then $u \leq v$ on $M$.
\end{thm}
\begin{proof}
Suppose that $u > v$ somewhere in the interior of $M$. Let $C$ be the maximum of $\frac{u}{v}$ on $M$, which is attained at $x_0$ in the interior of $M$. Since $C > 1$, by Theorem \ref{prop} we know that $w = \frac{u}{C}$ is a strict subsolution, that is,
\begin{equation*}
\sigma_k^{\frac{1}{k}} ( W[w] ) > \frac{n - 2}{2} w^{\frac{n + 2}{n - 2}}.
\end{equation*}
On the other hand, since $w(x_0) = v(x_0)$ while $w \leq v$ near $x_0$, thus at $x_0$ we have
\[\nabla w (x_0) = \nabla v (x_0),\quad \nabla^2 w (x_0) \leq \nabla^2 v (x_0)\]
and consequently $W[w](x_0) \leq W[v](x_0)$. It follows that
\[\sigma_k^{\frac{1}{k}} ( W[v] )(x_0) \geq \sigma_k^{\frac{1}{k}} ( W[w] ) (x_0) > \frac{n - 2}{2} w^{\frac{n + 2}{n - 2}} (x_0) = \frac{n - 2}{2} v^{\frac{n + 2}{n - 2}}(x_0), \]
contradicting with the fact that $v$ is a supersolution.
\end{proof}

Next, we provide some special solutions on balls and exterior of balls in $\mathbb{R}^n$ of equation \eqref{conformal2}, which will be used frequently as supersolutions or subsolutions.
In Euclidean coordinates, $W[u]$ can be expressed as
\begin{equation} \label{Wij}
W_{ij} [ u ] = (n - 2) u_{ij} - n \frac{u_i u_j}{u} + \Big( \Delta u + \frac{|\nabla u|^2}{u} \Big) \delta_{ij}.
\end{equation}
Let
\[ B_R(x_0) = \big\{ x \in \mathbb{R}^n  \big\vert  |x - x_0| < R \big\}. \]
We can verify the following proposition.

\begin{prop} \label{subsolution}
For any fixed $s > 0$,
\begin{equation*}
 u (x) =  \Big( 4 (n - 1) (C_n^k)^{\frac{1}{k}} s^2 \Big)^{\frac{n - 2}{4}} \Big( s^2 - {|x - x_0|}^2 \Big)^{1 - \frac{n}{2}} \quad \text{ in  }  B_s (x_0),
\end{equation*}
\begin{equation*}
v (x) = \Big( 4 (n - 1) (C_n^k)^{\frac{1}{k}}  s^2 \Big)^{\frac{n - 2}{4}} \Big( {|x - x_0|}^2 - s^2 \Big)^{1 - \frac{n}{2}} \quad \text{  in  } \mathbb{R}^n\setminus \overline{B_s (x_0)}
\end{equation*}
are admissible solutions of \eqref{conformal2} which approach to $\infty$ on $\partial B_s (x_0)$.
\end{prop}

\begin{proof}
We first prove that $u(x)$ is an admissible solution of \eqref{conformal2}.
For convenience, denote
\[ c: = \Big( 4 (n - 1) (C_n^k)^{\frac{1}{k}} s^2 \Big)^{\frac{n - 2}{4}}, \quad x = (x^1, \ldots, x^n), \quad x_0 = (x_0^1, \ldots, x_0^n). \]
Direct calculation shows that
\[ W_{ij} [u] = 2 (n - 1) (n - 2) c \Big( s^2 - |x - x_0|^2 \Big)^{- \frac{n}{2} - 1} s^2  \delta_{ij}. \]
Hence
\[ \sigma_k^{\frac{1}{k}}\big( W[u] \big) = 2 (n - 1) (n - 2) c \Big( s^2 - |x - x_0|^2 \Big)^{- \frac{n}{2} - 1} s^2 \big( C_n^k \big)^{\frac{1}{k}}, \]
which agrees with the right hand side
\[ \frac{n - 2}{2} u^{\frac{n + 2}{n - 2}} = \frac{n - 2}{2} c^{\frac{n + 2}{n - 2}} \Big( s^2 - |x - x_0|^2 \Big)^{ - \frac{n}{2} - 1 }. \]
$v(x)$ can be verified similarly.
\end{proof}

The following result is a direct consequence of Proposition \ref{subsolution} and Theorem \ref{MP}.
\begin{cor}
For any positive admissible solution $u$ of \eqref{conformal2} in $B_R(x_0)$, we have
\begin{equation} \label{ball}
 u(x) \leq \Big( 4 (n - 1) (C_n^k)^{\frac{1}{k}} R^2 \Big)^{\frac{n - 2}{4}} \Big( R^2 - {|x - x_0|}^2 \Big)^{1 - \frac{n}{2}} \quad\text{  in  } B_R(x_0).
\end{equation}
In particular, we have
\begin{equation} \label{eq1_cor}
 u(x_0) \leq  \Big( 4 (n - 1) (C_n^k)^{\frac{1}{k}} \Big)^{\frac{n - 2}{4}} R^{1 - \frac{n}{2}}.
\end{equation}
\end{cor}

As a consequence of \eqref{eq1_cor}, we immediately obtain
\begin{cor}
There does not exist a positive admissible solution to \eqref{conformal2} on $\mathbb{R}^n$.
\end{cor}
\begin{proof}
If $u$ is a positive admissible solution to \eqref{conformal2} on $\mathbb{R}^n$, then $u$ is obviously a positive admissible solution on $B_R(0)$ for any $R > 0$. By \eqref{eq1_cor} we have
\[  u(0) \leq  \Big( 4 (n - 1) (C_n^k)^{\frac{1}{k}} \Big)^{\frac{n - 2}{4}} R^{1 - \frac{n}{2}}. \]
Letting $R \rightarrow \infty$ yields
\[ u (0) \leq 0, \]
which contradicts with the fact that $u (0) > 0$.
\end{proof}

We also find the decay property at $\infty$ of any positive admissible solution of \eqref{conformal2}.
\begin{cor} \label{decay}
Let $\Omega$ be an unbounded domain in $\mathbb{R}^n$ with smooth compact boundary.
If $u$ is a smooth positive admissible solution of \eqref{conformal2}, then
\begin{equation} \label{decaying}
u(x) \rightarrow 0 \quad \text{ as } \,\, |x| \rightarrow \infty.
\end{equation}
\end{cor}
\begin{proof}
Applying \eqref{eq1_cor} in a ball centered at $x$ of radius $|x|/2$ for $|x|$ large, we have
\[ u(x) \leq  \Big( 4 (n - 1) \big( C_n^k \big)^{\frac{1}{k}} \Big)^{\frac{n - 2}{4}} \Big( {\frac{|x|}{2}} \Big)^{1 - \frac{n}{2}},  \]
which implies \eqref{decaying}.
\end{proof}

\vspace{2mm}

\subsection{Growth rate near codimension $1$ boundary}~

\vspace{2mm}

In this section, we assume that the boundary of $\Omega$ is composed of smooth compact hypersurfaces. For positive admissible solution of \eqref{conformal2} which approaches $\infty$ at $\partial\Omega$, we characterize the growth rate.

\begin{thm} \label{asymptotic_near_boundary}
Let $(M, g)$ be a smooth manifold with boundary. Assume that $\partial M$ is compact. If $u$ is a smooth positive admissible solution of \eqref{conformal2} on $M$ which approaches to $\infty$ at $\partial M$, then
\begin{equation} \label{limboundary}
{\rho}^{ \frac{n}{2} - 1}(x) u(x) \rightarrow  \big( (n - 1) (C_n^k)^{\frac{1}{k}} \big)^{\frac{n - 2}{4}} \quad \text{  as  } x \rightarrow \partial M,
\end{equation}
where $\rho(x)$ is the distance from $x$ to $\partial M$.
\end{thm}
\begin{proof}
We first construct a positive admissible subsolution in a neighborhood of $\partial M$. Consider
\[ \underline{u} = c_0  (\rho + \epsilon)^{1 - \frac{n}{2}} e^w, \]
where $\epsilon$ is a small positive constant,
\[ c_0 = \big( (n - 1) (C_n^k)^{\frac{1}{k}} \big)^{\frac{n - 2}{4}}, \]
and
\[ w = w(\rho) =  \frac{1}{\rho + \delta} - \frac{1}{\delta} \]
with $\delta$ a small positive constant to be chosen later.

For any point $x \in M \setminus \partial M$ at which $\rho$ is smooth, we choose a local orthonormal frame $e_1, \ldots, e_n$ around $x$ such that $e_1 = \frac{\partial}{\partial \rho}$ (recall that $\nabla \rho$ = $\frac{\partial}{\partial \rho}$ and $|\nabla \rho| = 1$ at $x$). Elementary calculation shows that at $x$,
\[ \nabla \underline{u} = c_0 e^w  (\rho + \epsilon)^{- \frac{n}{2}} \Big( 1 - \frac{n}{2} + (\rho + \epsilon) w' \Big) \nabla \rho, \]

\begin{equation*}
\begin{aligned}
\nabla^2 \underline{u} = & c_0 e^w (\rho + \epsilon)^{- \frac{n}{2} - 1} \Bigg( \bigg( \Big( 1 - \frac{n}{2} \Big)(\rho + \epsilon) + (\rho + \epsilon)^2 w' \bigg) \nabla^2 \rho  \\
& + \bigg( \frac{n}{2} \Big( \frac{n}{2} - 1 \Big) + (\rho + \epsilon)^2 w'' - (n - 2) (\rho + \epsilon) w' + (\rho + \epsilon)^2 w'^2 \bigg) \nabla \rho \otimes \nabla \rho  \Bigg),
\end{aligned}
\end{equation*}
and consequently,
\begin{equation*}
\begin{aligned}
W[ \underline{u} ] = & c_0 e^w (\rho + \epsilon)^{ - \frac{n}{2} - 1} \Bigg( \Big( (n - 2)(\rho + \epsilon)^2 w' - \frac{1}{2}(n - 2)^2 (\rho + \epsilon) \Big) \nabla^2 \rho \\
& + \Big( (n - 2)(\rho + \epsilon)^2 w'' + 2 (n - 2)(\rho + \epsilon) w' - 2 (\rho + \epsilon)^2 w'^2 \Big) \nabla \rho \otimes \nabla \rho \\
& + \bigg( \Big( (\rho + \epsilon)^2 w' - \frac{1}{2}(n - 2)(\rho + \epsilon) \Big) \Delta \rho \\
& + (\rho + \epsilon)^2 w'' - 2 (n - 2) (\rho + \epsilon) w' + 2 (\rho + \epsilon)^2 w'^2 + \frac{1}{2}(n - 1)(n - 2) \bigg) I \\
& - \frac{n - 2}{2} (\rho + \epsilon)^2 \text{Ric}_g \Bigg).
\end{aligned}
\end{equation*}

For $\rho_0 > 0$, denote
\[ M_{\rho_0} = \big\{ x \in M  \big\vert  \rho(x) < \rho_0 \big\}. \]
We may choose $\rho_0$ sufficiently small such that on $M_{\rho_0}$, $u > c_0$ and
\begin{equation}\label{eq3}
\Delta \rho I + (n - 2) \nabla^2 \rho \leq \lambda I,  \quad \text{Ric}_g \leq C_0 I,
\end{equation}
where $I$ is the unit matrix and $\lambda$, $C_0$ are positive constants depending only on $g = I$.
Denoting
\begin{equation*}
\begin{aligned}
\Psi = & (n - 2)(\rho + \epsilon)^2 w'' \nabla \rho \otimes \nabla \rho  + (n - 2)(\rho + \epsilon)^2 w' \nabla^2 \rho + (\rho + \epsilon)^2 w'' I \\
& + 2 (\rho + \epsilon)^2  w'^2 (I - \nabla \rho \otimes \nabla \rho) + (\rho + \epsilon)^2 w' \Delta \rho I - \frac{n - 2}{2} (\rho + \epsilon)^2 \text{Ric}_g,
\end{aligned}
\end{equation*}
\begin{equation*}
\begin{aligned}
\Phi = & - \frac{1}{2} (n - 2)(\rho + \epsilon) \Delta \rho I - \frac{1}{2}(n - 2)^2 (\rho + \epsilon) \nabla^2 \rho \\
& + 2 (n - 2) (\rho + \epsilon) w' (\nabla \rho \otimes \nabla \rho - I) + \frac{1}{2}(n - 1)(n - 2) I,
\end{aligned}
\end{equation*}
then
\[ W[\underline{u}] =  c_0 e^w (\rho + \epsilon)^{ - \frac{n}{2} - 1} (\Psi + \Phi).  \]
Note that
\[ w' = \frac{- 1}{(\rho + \delta)^2} < 0, \quad
w'' = \frac{2}{(\rho + \delta)^3} > 0. \]
In view of \eqref{eq3} we obtain
\begin{equation*}
\Psi \geq (\rho + \epsilon)^2 \Big( w'' +  \lambda w' - \frac{n - 2}{2} C_0 \Big) I.
\end{equation*}
We note that
\begin{equation*}
w'' +  \lambda w' - \frac{n - 2}{2} C_0
= \frac{2}{(\rho + \delta)^3} - \frac{\lambda}{(\rho + \delta)^2} - \frac{n - 2}{2} C_0 \geq 0
\end{equation*}
for $\rho_0 + \delta $ sufficiently small depending only on $n$, $\lambda$ and $C_0$. Hence we conclude that $\Psi \geq 0$.

Also, by \eqref{eq3} and for $\rho_0 + \delta \leq \sqrt{\frac{2}{\lambda}}$, we have
\begin{equation*}
\begin{aligned}
\Phi \geq & - \frac{n - 2}{2}(\rho + \epsilon) \lambda I + \frac{2(n - 2)(\rho + \epsilon)}{(\rho + \delta)^2} (I - \nabla \rho \otimes \nabla \rho) + \frac{1}{2} (n - 1) (n - 2) I
\\ \geq & \frac{1}{2} (n - 1)(n - 2) \left(
                        \begin{array}{cccc}
                          1 - \frac{\lambda (\rho + \epsilon)}{n - 1} & \, & \, & \, \\
                          \, & 1 + \frac{\lambda (\rho + \epsilon)}{n - 1} & \, & \, \\
                          \, & \, & \ddots & \, \\
                          \, & \, & \, & 1 + \frac{\lambda (\rho + \epsilon)}{n - 1} \\
                        \end{array}
                      \right).
\end{aligned}
\end{equation*}
Letting $\epsilon < \rho_0$ and $\rho_0 < \frac{n - 1}{2 \lambda}$, we have $\Phi > 0$ and
\begin{equation*}
\begin{aligned}
\sigma_k^{\frac{1}{k}} ( W[\underline{u}] ) \geq & c_0 e^w (\rho + \epsilon)^{ - \frac{n}{2} - 1} \sigma_k^{\frac{1}{k}} ( \Phi ) \\
\geq & \frac{1}{2}(n - 1)(n - 2) c_0 e^w (\rho + \epsilon)^{- \frac{n}{2} - 1}
( C_n^k )^{\frac{1}{k}} \cdot \\
& \Big( 1 + \frac{\lambda( \rho + \epsilon )}{n - 1} \Big)^{\frac{k - 1}{k}} \bigg( 1 + \Big( 1 - \frac{2 k}{n} \Big) \frac{\lambda ( \rho + \epsilon)}{n - 1} \bigg)^{\frac{1}{k}}.
 \end{aligned}
 \end{equation*}
Elementary calculation shows that
\[ \Big( 1 + \frac{ \lambda (\rho + \epsilon)}{n - 1} \Big)^{k - 1} \bigg( 1 + \Big( 1 - \frac{2 k}{n} \Big) \frac{\lambda (\rho + \epsilon)}{n - 1} \bigg) \geq  1 \]
when $k \leq \frac{n}{2}$, or when $k > \frac{n}{2}$ and $\rho_0 \leq \frac{(n - 1)(n - 2) k}{2 (2 k - n) (k - 1) \lambda}$.
Therefore,
\begin{equation*}
\sigma_k^{\frac{1}{k}} ( W[\underline{u}] ) \geq \frac{1}{2} (n - 1) (n - 2) c_0 (C_n^k)^{\frac{1}{k}} e^w (\rho + \epsilon)^{- \frac{n}{2} - 1}.
\end{equation*}
On the other hand,
\[ \frac{n - 2}{2} \underline{u}^{\frac{n + 2}{n - 2}} = \frac{n - 2}{2} {c_0}^{\frac{n + 2}{n - 2}} (\rho + \epsilon)^{- \frac{n}{2} - 1} e^{\frac{n + 2}{n - 2} w}. \]
Note that $w \leq 0$. Hence $\underline{u}$ is a subsolution of \eqref{conformal2}.

Observe that
\[ \underline{u} = c_0 \epsilon^{1 - \frac{n}{2}} \leq u = \infty \quad \text{  on  } \partial M, \]
and on $\rho = \rho_0$,
\begin{equation*} \label{on level surface}
\underline{u} \leq c_0 {\rho_0}^{1 - \frac{n}{2}} e^{(\rho_0 + \delta)^{- 1} - {\delta}^{- 1} }
\leq c_0 {\rho_0 }^{1 - \frac{n}{2}} e^{- \delta^{- 1} / 2 }  \leq  c_0 \leq  u
\end{equation*}
for $\delta$ sufficiently small depending in addition on $\rho_0$.
By the maximum principle,
\[ \underline{u}  \leq  u  \quad \text{  on  } M_{\rho_0}.  \]
Letting $\epsilon \rightarrow 0^+$ and then $x \rightarrow \partial M$ we obtain
\[ \liminf\limits_{x \rightarrow \partial M} \rho^{\frac{n}{2} - 1} u  \geq  c_0. \]

We next prove
\begin{equation} \label{limsup}
\limsup\limits_{x \rightarrow \partial M} \rho^{\frac{n}{2} - 1} u  \leq  c_0.
\end{equation}
For this, we construct a supersolution in $M_{\rho_1} \setminus \overline{M_{\epsilon}}$ for any $0 < \epsilon < \rho_1$:
\begin{equation*}
\overline{u} = c_0  (\rho - \epsilon)^{1 - \frac{n}{2}} e^v,
\end{equation*}
where
\[ v = v(\rho) =  \frac{1}{2} \Big( \ln (\delta + \rho) - \ln \delta \Big)  \]
with $\rho_1$, $\delta$  small positive constants to be chosen later.

From the above calculation, it is easy to obtain
\begin{equation*}
\begin{aligned}
\sigma_1 \big( W[ \overline{u} ] \big) = & c_0 e^v (\rho - \epsilon)^{ - \frac{n}{2} - 1} \bigg( \Big( 2 (n - 1)(\rho - \epsilon)^2 v' - (n - 2)(n - 1) (\rho - \epsilon) \Big) \Delta \rho \\
&+ 2 (n - 1)(\rho - \epsilon)^2 v'' - 2 (n - 2) (n - 1) (\rho - \epsilon) v' + 2 (n - 1) (\rho - \epsilon)^2 v'^2 \\
&+ \frac{n (n - 1) (n - 2)}{2} - \frac{n - 2}{2} (\rho - \epsilon)^2 S_g \bigg),
\end{aligned}
\end{equation*}
where $S_g$ is the scalar curvature on $M$.
We may choose $\rho_1$ sufficiently small such that on $M_{\rho_1}$,
\[ | \Delta \rho | \leq C_1,  \quad S_g \geq - C_1,
\]
where $C_1$ is a positive constant depending only on $g$.
Note that
\[ v' = \frac{1}{2 (\delta + \rho)}, \quad v'' = - \frac{1}{2 (\delta + \rho)^2}. \]
It follows that
\begin{equation*}
\begin{aligned}
& 2 (n - 1) v' \Delta \rho + 2 (n - 1) v'' + 2 (n - 1) v'^2 - \frac{n - 2}{2} S_g \\
\leq & \frac{(n - 1) C_1}{\delta + \rho} - \frac{n - 1}{2 (\delta + \rho)^2} + \frac{n - 2}{2} C_1 \leq 0,
\end{aligned}
\end{equation*}
and
\begin{equation*}
\begin{aligned}
 - (n - 2)(n - 1) \Delta \rho - 2 (n - 2) (n - 1) v'
\leq  (n - 2)(n - 1) \Big( C_1 - \frac{1}{\delta + \rho} \Big) \leq 0
\end{aligned}
\end{equation*}
for $\rho_1 + \delta$ sufficiently small depending only on $n$ and $C_1$.
Consequently,
\begin{equation*}
\sigma_1 \big( W[ \overline{u} ] \big) \leq c_0 e^v (\rho - \epsilon)^{ - \frac{n}{2} - 1}
 \frac{n (n - 1) (n - 2)}{2}.
\end{equation*}
Then applying Newton-Maclaurin inequality we arrive at
\[ \sigma_k^{\frac{1}{k}} \big( W[ \overline{u} ] \big) \leq \frac{(C_n^k)^{\frac{1}{k}}}{n} \sigma_1 \big( W[ \overline{u} ] \big) \leq \frac{n - 2}{2} {\overline{u}}^{\frac{n + 2}{n - 2}} \quad \text{  in  } M_{\rho_1} \setminus \overline{M_{\epsilon}}. \]

Also, we note that on $\rho = \epsilon$,
\[   \overline{u} = \infty > u,    \]
and on $\rho = \rho_1$,
\[ \overline{u} \geq c_0  {\rho_1}^{1 - \frac{n}{2}} \sqrt{\frac{\delta + \rho_1}{\delta}} \geq u
\]
if $\delta$ is chosen sufficiently small depending in addition on $\rho_1$ and $u$.
By the maximum principle, we have
\[ \overline{u}  \geq  u  \quad \text{  on  } M_{\rho_1} \setminus M_{\epsilon}.  \]
Letting $\epsilon \rightarrow 0^+$ and then $x \rightarrow \partial M$ we obtain
\eqref{limsup}.
\end{proof}

\begin{rem}
In \cite{Gursky-Streets-Warren}, a global subsolution
\[ \underline{u} = c_0 (\rho + \epsilon)^{1 - \frac{n}{2}} e^{A \big( (\rho + \delta)^{- p} - {\delta}^{- p} \big)} \]
is constructed on compact manifold with boundary which yields the lower bound of the growth rate. We simplify the proof by defining the subsolution only in a neighborhood of the boundary $\partial M$. For the upper bound, we also adopt a new supersolution near $\partial M$. For the special case when $M$ is a domain  in $\mathbb{R}^n$, we may use Gonz\'alez-Li-Nguyen's method (see Lemma 3.4 in \cite{Gonzalez-Li-Nguyen}) to give an easier proof.
\end{rem}

\vspace{2mm}

\subsection{Existence of smooth solutions}~

\vspace{2mm}

We now cite an important result of Loewner and Nirenberg \cite{Loewner} when the boundary of the domain is composed of hypersurfaces for the scalar case ($k = 1$), which will be used as an upper barrier later.
\begin{lemma}  \label{Loewner}
Let $\Omega$ be a domain in $\mathbb{R}^n$ having smooth compact hypersurfaces as boundary. There exists a unique positive solution $u \in C^{\infty}(\Omega)$ of \eqref{scalar_eqn}. In addition, if $\rho(x)$ denotes the distance of $x$ to the boundary of $\Omega$, then
\[{\rho}^{ \frac{n}{2} - 1}(x) u(x) \rightarrow  \big( n (n - 1) \big)^{\frac{n - 2}{4}}\quad \text{ as }  x \rightarrow \partial\Omega. \]
If $\Omega$ is unbounded, the solution $u(x)$ has the property that
\[ |x|^{n - 2} u(x) \rightarrow c \quad \text{  as  } |x| \rightarrow \infty, \]
where $c$ is a positive constant depending only on $n$ and $\Omega$.
\end{lemma}

We can now state the main theorem in this section.

\begin{thm} \label{Second Half}
Let $\Omega$ be a domain in $\mathbb{R}^n$ with nonempty smooth compact boundary which are composed of closed hypersurfaces.
Then there exists a positive admissible solution $u \in C^{\infty} (\Omega)$ to equation \eqref{conformal2} which tends to $\infty$ on $\partial \Omega$. If $\Omega$ is bounded or $k = 1$, the solution is unique.
\end{thm}
\begin{proof}
If $\Omega$ is a bounded domain, by Corollary 1.3 in \cite{Guan08}, there exists a positive admissible solution $u \in C^{\infty} (\Omega)$ to equation \eqref{conformal2} with infinite boundary value.

When $\Omega$ is unbounded, we may assume without loss of generality that $0 \notin \overline{\Omega}$. Let $B_s(0)$ be a fixed ball such that $\overline{\Omega} \subset \mathbb{R}^n\setminus \overline{B_s(0)}$. By Proposition \ref{subsolution}, there exists a smooth solution
\[ \underline{u} = C_1 (|x|^2 - s^2)^{1 - \frac{n}{2}}, \quad C_1 = \big( 4 (n - 1) (C_n^k)^{\frac{1}{k}}  s^2 \big)^{\frac{n - 2}{4}}  \]
to equation \eqref{conformal2} in $\mathbb{R}^n \setminus \overline{B_s(0)}$.

For any $R > \max\limits_{\partial\Omega} \underline{u}$ large enough such that $\partial\Omega \subset B_R(0)$, by Corollary 5.4 in \cite{Guan08}, there exists a smooth positive admissible solution $u_R$ to the Dirichlet problem
\begin{equation}
\left\{
\begin{aligned}
& \sigma_k^{\frac{1}{k}}(W[u]) =  \frac{n - 2}{2} {u}^{\frac{n + 2}{n - 2}}  \quad \text{  in  }   B_{R}(0) \cap \Omega, \\
& u =  R \quad \text{  on  }  \partial \Omega, \\
& u =  \underline{u} \quad \text{  on  }  \partial B_{R}(0).
\end{aligned}
\right.
\end{equation}

By Lemma \ref{Loewner}, we are able to find a positive solution $\bar u \in C^{\infty}(\Omega)$ of
\eqref{scalar_eqn} which tends to $\infty$ at $\partial \Omega$ with the growth rate
\[ \rho^{\frac{n}{2} - 1} (x) \overline{u}(x) \rightarrow  \big( n (n - 1) \big)^{\frac{n - 2}{4}} \quad \text{  as  }  x \rightarrow \partial \Omega, \]
and decays to $0$ at $\infty$ with the decay rate
\[  |x|^{n - 2} \overline{u} (x) \rightarrow  c \quad \text{  as  } |x| \rightarrow \infty, \]
where $c$ is the positive constant in Lemma \ref{Loewner}.

Comparing the decay rate as $|x| \rightarrow \infty$, we have on $\partial B_R(0)$ with $R$ sufficiently large,
\[ u_R (x) = \underline{u}(x) \leq \frac{2 C_1}{c} \overline{u}(x). \]
Meanwhile, we have
\[ u_R (x) = R < \infty = \frac{2 C_1}{c} \overline{u}(x) \quad \text{  on  } \partial \Omega. \]
By the maximum principle,  we obtain
\[ u_R(x) \leq \max\Big\{\frac{2 C_1}{c}, 1 \Big\} \overline{u} (x) \quad \text{  on  } B_R(0) \cap \Omega. \]
In proving the above inequality, we have applied the Newton-Maclaurin inequality as well as Proposition \ref{prop}.
On the other hand, by the maximum principle we have
\[  u_R(x) \geq \underline{u}(x) \quad \text{  on  }  B_R(0) \cap \Omega.   \]
Now we can apply the interior estimates of Guan (see Theorems 2.1 and 3.1 in \cite{Guan08}), followed by the Evans-Krylov interior estimates \cite{Evans, Krylov} and a standard diagonal process to obtain a smooth positive admissible solution $u$ of \eqref{conformal2} on $\Omega$ which tends to $\infty$ at $\partial \Omega$.

When $\Omega$ is bounded, we shall prove the uniqueness. Suppose $v$ is another positive admissible solution satisfying \eqref{conformal2}-\eqref{eq1-2}. By Theorem \ref{asymptotic_near_boundary}, $u$ and $v$ have the same growth rate as $x\rightarrow \partial \Omega$. Hence for any $\epsilon > 0$, on $\{ \rho = c \}$ where $c > 0$ is a sufficiently small positive constant, we have
\[ u \leq  (1 + \epsilon) v. \]
By Proposition \ref{prop} and the maximum principle,
\[ u  \leq ( 1 + \epsilon) v \quad \text{ in }\,\, \{ x \in \Omega |  \rho (x) > c \}. \]
Letting $c \rightarrow 0$, we have
\[ u  \leq  (1 + \epsilon)  v \quad \text{ in } \,\, \Omega. \]
Then let $\epsilon \rightarrow 0$. It follows that
\[  u  \leq  v \quad \text{ in } \, \, \Omega. \]
Similarly, we can show that $v \leq u$ in $\Omega$. Hence we proved the uniqueness.
\end{proof}

\vspace{4mm}

\section{Maximal solution on general domain in $\mathbb{R}^n$}

\vspace{4mm}

For any domain $\Omega\subset\mathbb{R}^n$ with smooth compact boundary $\partial \Omega$, we hope to associate a smooth positive admissible solution $u_{\Omega}$ of \eqref{conformal2} which is maximal, in the sense that it is greater than or equal to any smooth positive admissible solution of \eqref{conformal2} in $\Omega$. We will then investigate when $u_\Omega$ tends to $\infty$ on $\partial \Omega$.

\vspace{2mm}

\subsection{Construction of the maximal solution}~

\vspace{2mm}

Let $\Omega_{(1)} \Subset \Omega_{(2)} \Subset \ldots$ be an increasing sequence of bounded subdomains of $\Omega$ with smooth compact boundaries $\partial\Omega_{(j)}$ which are closed hypersurfaces such that $\Omega = \cup \Omega_{(j)}$.
By Theorem \ref{Second Half}, we can find a unique smooth positive admissible solution $u_{(j)}$ of \eqref{conformal2} in $\Omega_{(j)}$ which tends to $\infty$ on the boundary $\partial \Omega_{(j)}$.
By the maximum principle, we see that $\{u_{(j)}\}$ is a monotone decreasing sequence of positive functions. It follows that $u_{(j)}$ converges to a nonnegative function $u_{\Omega}$ in $\Omega$.

We have the following dichotomy, which is similar to \cite{Gonzalez-Li-Nguyen}.

\begin{lemma} \label{dichotomy}
We have either $u_{\Omega} > 0$ in $\Omega$ or $u_{\Omega} \equiv 0$ in $\Omega$.
\end{lemma}

\begin{proof}
For the sake of completeness, we provide the proof. Suppose that $u_{\Omega} \not\equiv 0$. Then we must have $u_{\Omega} > c$ on some $B(x_0, r_0) \subset \Omega$ for some constant $c > 0$. We may choose $0 < r_1 < r_0$ such that the solution
\[ v = \big( 4 (n - 1) (C_n^k)^{\frac{1}{k}}  r_1^2 \big)^{\frac{n - 2}{4}} \big( {|x - x_0|}^2 - r_1^2 \big)^{1 - \frac{n}{2}} \quad \text{  in  } \mathbb{R}^n\setminus \overline{B_{r_1} (x_0)} \]
(see Proposition \ref{subsolution}) satisfies
\[  \big( 4 (n - 1) (C_n^k)^{\frac{1}{k}}  r_1^2 \big)^{\frac{n - 2}{4}} \big( r_0^2 - r_1^2 \big)^{1 - \frac{n}{2}} = c.  \]
Let $\Omega_{(j)}$ and $u_{(j)}$ be as above. By the maximum principle, we know that $u_{(j)} \geq v$ in $\Omega_{(j)}\setminus B(x_0, r_0)$ for any $j$. Hence $u_{\Omega} \geq v > 0$ in $\Omega\setminus B(x_0, r_0)$. We thus have $u_{\Omega} > 0$ in $\Omega$.
\end{proof}

\begin{rem}
Several remarks are as follows.
\begin{enumerate}
\item In view of Lemma \ref{dichotomy} and the interior regularity of Guan \cite{Guan08} and Evans-Krylov \cite{Evans, Krylov}, we know that $u_{\Omega}$ is smooth.

\item If there exists a positive admissible subsolution $v$ of \eqref{conformal2} in $\Omega$ (for example, when $\mathbb{R}^n \setminus \overline{\Omega} \neq \emptyset$), then $u_{\Omega} > 0$ in $\Omega$.

\item When $u_{\Omega} > 0$ in $\Omega$, $u_{\Omega}$ is the so-called maximal solution of equation \eqref{conformal2} in $\Omega$. In fact, if $w$ is any smooth positive admissible solution of \eqref{conformal2} in $\Omega$, then $u_{(j)} \geq w$ on $\Omega_{(j)}$ by the maximum principle. Hence $u_{\Omega} \geq  w$ in $\Omega$.

\item When $\Omega$ is bounded and $\partial\Omega$ is composed of closed hypersurfaces, $u_{\Omega}$ coincides with the solution given by Theorem \ref{Second Half}.
\end{enumerate}
\end{rem}

As in \cite{Loewner}, we call a compact subset $\Gamma \subset \partial\Omega$ regular, if
\begin{equation} \label{regular}
u_{\Omega} (x) \rightarrow \infty  \quad \text{  as  }  x \rightarrow \Gamma.
\end{equation}
The rest of this section discusses the validity of \eqref{regular}.

\vspace{2mm}

\subsection{Regularity of a portion on the boundary}~

\vspace{2mm}

We consider a portion $\Gamma\subset\partial\Omega$ which is a smooth compact non-self-intersecting surface of codimension $m$. We also assume that $\partial\Omega\setminus\Gamma$ is smooth compact.
As in \cite{Loewner, Gonzalez-Li-Nguyen}, we first give a necessary and sufficient condition for \eqref{regular}.

\begin{thm} \label{nscondition for regularity}
Let $\Omega$ be a domain in $\mathbb{R}^n$ and $\Gamma$ be a compact subset of $\partial\Omega$ such that $\partial\Omega\setminus\Gamma$ is also compact. Suppose that $u_{\Omega} \not\equiv 0$. Then
$u_{\Omega}(x) \rightarrow \infty$ as $x \rightarrow \Gamma$ if and only if there exists an open neighborhood $U$ of $\Gamma$ and a $C^2$ positive admissible subsolution $\phi(x)$ of \eqref{conformal2} defined in $\Omega \cap U$ which tends to $\infty$ as $x \rightarrow \Gamma$.
\end{thm}

\begin{proof}
Necessity is obvious. For sufficiency, without loss of generality we may assume that $\overline{U}$ is compact, $\overline{U} \cap (\partial \Omega \setminus \Gamma) = \emptyset$ and $0 < \phi \in C^2(\overline{U} \cap \Omega)$. For $j$ sufficiently large, we have
\[ \partial (U \cap \Omega_{(j)}) = \partial U  \cup ( U \cap \partial \Omega_{(j)} ). \]
Since $u_{\Omega}$ is positive in $\Omega$ and $\partial U \subset \Omega$ is compact, we have
$\inf\limits_{\partial U} u_{\Omega} := m > 0$.
Denote $M := \sup\limits_{\partial U} \phi(x) > 0$ and $A := \max\{\frac{M}{m}, 1\}$. We note that
$u_{(j)} \geq u_{\Omega} \geq m \geq \frac{\phi}{A}$ on $\partial U$, and
$u_{(j)} = \infty > \frac{\phi}{A}$ on $U \cap \partial \Omega_{(j)}$.
By Proposition \ref{prop}, $\frac{\phi}{A}$ is again a positive admissible subsolution of \eqref{conformal2}. By the maximum principle, we arrive at
$u_{(j)} \geq \frac{\phi}{A}$ in $U \cap \Omega_{(j)}$.
Consequently, we obtain $u_{\Omega} \geq \frac{\phi}{A}$ in $U \cap \Omega$.
Hence we proved the sufficiency.
\end{proof}

Now we aim at constructing $\phi(x)$ in a neighborhood of $\Gamma$ as in Theorem \ref{nscondition for regularity}. Let $\rho(x)$ denote the distance of $x$ to $\Gamma$. Define
\[ \Gamma_{\rho_0}  = \{ x \in \Omega | \rho(x) < \rho_0 \}. \]
For $\rho_0$ sufficiently small, $\rho$ is a smooth function in $\Gamma_{\rho_0} \setminus \Gamma$. Define on $\Gamma_{\rho_0} \setminus \Gamma$
\[ \phi(x) = c \rho^{1 - \frac{n}{2}} (x), \]
where $c > 0$ is a constant to be determined later.
Recall that $|\nabla \rho| = 1$. Direct calculation shows that
\[ \phi_i = c \Big( 1 - \frac{n}{2} \Big) {\rho}^{- \frac{n}{2}} \rho_i, \]
\[ \phi_{ij} = c \Big( \frac{n}{2} - 1 \Big) \frac{n}{2} {\rho}^{ - \frac{n}{2} - 1} \rho_i \rho_j + c \Big( 1 - \frac{n}{2} \Big) {\rho}^{- \frac{n}{2}} \rho_{ij}, \]
\[ |\nabla \phi|^2 = c^2 \Big( 1 - \frac{n}{2} \Big)^2 {\rho}^{- n}, \]
\[ \Delta \phi = c \Big( \frac{n}{2} - 1 \Big) \frac{n}{2} {\rho}^{- \frac{n}{2} - 1} + c \Big( 1 - \frac{n}{2} \Big) {\rho}^{- \frac{n}{2}} \Delta \rho. \]
Thus, we have
\begin{equation*}
W_{ij}[\phi] = \frac{1}{2}(n - 2) c {\rho }^{- \frac{n}{2} - 1} \Big( (n - 1 - \rho \Delta \rho) \delta_{ij} - (n - 2) \rho \rho_{ij} \Big).
\end{equation*}
As shown in \cite{Gonzalez-Li-Nguyen}, as $\rho \rightarrow 0$,
\[ \lambda(\rho \nabla^2\rho) \rightarrow \big( \underbrace{0, \ldots, 0}_{n - m + 1}, \underbrace{1, \ldots, 1}_{m - 1} \big). \]
Consequently, as $\rho \rightarrow 0$,
\begin{equation*}
\lambda\big( {\rho}^{ \frac{n}{2} + 1} W_{ij}[\phi] \big) \rightarrow \frac{1}{2}(n - 2) c \big( \underbrace{n - m, \ldots, n - m}_{n - m + 1}, \underbrace{2 - m, \ldots, 2 - m}_{m - 1} \big).
\end{equation*}
If we assume that
\[ v_m = \big( \underbrace{n - m, \ldots, n - m}_{n - m + 1}, \underbrace{2 - m, \ldots, 2 - m}_{m - 1} \big) \in \Gamma_k,  \]
then we have as $\rho \rightarrow 0$,
\[ {\rho}^{ \frac{n}{2} + 1} \sigma_k^{\frac{1}{k}} \big( W_{ij} [\phi] \big) \rightarrow \frac{1}{2} (n - 2) c \sigma_k^{\frac{1}{k}}( v_m ). \]
On the other hand,
\begin{equation*}
\frac{n - 2}{2} \phi^{\frac{n + 2}{n - 2}} = \frac{n - 2}{2} c^{\frac{n + 2}{n - 2}} \rho^{- \frac{n}{2} - 1}.
\end{equation*}
If we choose
\[ 0 < c < \Big( \sigma_k^{\frac{1}{k}} ( v_m ) \Big)^{\frac{n - 2}{4}}, \]
then $\phi$ is the one as in Theorem \ref{nscondition for regularity}, which implies that $u_{\Omega} \rightarrow \infty$ as $x \rightarrow \Gamma$ if $u_{\Omega} \not\equiv 0$.

\vspace{2mm}

\subsection{The case when $v_m \notin \overline{\Gamma}_k$}~

\vspace{2mm}

We adopt the same function as \cite{Gonzalez-Li-Nguyen}
\[ \psi = \psi_{c, d} = ( c \rho^{- \alpha} + d )^\beta, \]
where $\alpha$, $\beta$, $c$ and $d$ are positive constants. Recall that $\rho(x)$ is the distance of $x$ to $\Gamma$.
By direct calculation, we have
\[ \begin{aligned}
\psi_i = & - \alpha \beta c ( c \rho^{- \alpha} + d )^{\beta - 1} \rho^{- \alpha - 1} \rho_i, \\
\psi_{ij} = &  \alpha^2 \beta (\beta - 1) c^2 (c \rho^{- \alpha} + d)^{\beta - 2} \rho^{- 2 \alpha - 2} \rho_i \rho_j \\
& + \alpha (\alpha + 1) \beta c (c \rho^{- \alpha} + d)^{\beta - 1} \rho^{- \alpha - 2} \rho_i \rho_j
 - \alpha \beta c (c \rho^{- \alpha} +  d)^{\beta - 1} \rho^{- \alpha - 1} \rho_{ij}, \\
|\nabla\psi|^2 = & \alpha^2 \beta^2 c^2 ( c \rho^{- \alpha} + d)^{2 \beta - 2} \rho^{- 2 \alpha - 2}, \\
\Delta\psi = &  \alpha^2 \beta (\beta - 1) c^2 (c \rho^{- \alpha} + d)^{\beta - 2} \rho^{- 2 \alpha - 2} + \alpha (\alpha + 1) \beta c (c \rho^{-\alpha} + d)^{\beta - 1} \rho^{- \alpha - 2} \\
& - \alpha \beta c (c \rho^{- \alpha} +  d)^{\beta - 1} \rho^{- \alpha - 1} \Delta \rho. \\
\end{aligned}
\]
It follows that
\begin{equation*}
\begin{aligned}
\psi^{ - \frac{n + 2}{n - 2}} W_{ij} [ \psi ] = & \alpha \beta c \rho^{- \alpha - 2} (c \rho^{- \alpha} +  d)^{- \frac{4 \beta}{n - 2} - 1} \bigg( - (n - 2) \rho \rho_{ij} \\
& + \Big( \alpha (- n  - 2 \beta + 2 ) \frac{ c \rho^{- \alpha}}{c \rho^{- \alpha} + d} + (n - 2) (\alpha + 1) \Big)  \rho_i \rho_j \\
& + \Big( \alpha (2 \beta - 1) \frac{ c \rho^{- \alpha}}{c \rho^{- \alpha} + d} + \alpha + 1 - \rho \Delta \rho \Big)  \delta_{ij} \bigg).
\end{aligned}
\end{equation*}
Choosing an appropriate coordinate system as in \cite{Gonzalez-Li-Nguyen} such that
as $\rho\rightarrow 0$,
\[ \rho  \nabla^2 \rho = \left(
                                    \begin{array}{ccc}
                                      \mathcal{O}(\rho)_{(n - m)\times(n - m)} & \, & \, \\
                                      \, & 0 & \, \\
                                      \, & \, & I_{(m - 1)\times(m-1)} \\
                                    \end{array}
                                  \right)
 \]
and
\[ \nabla\rho \otimes \nabla\rho = \left(
                                               \begin{array}{ccc}
                                                 0_{(n - m) \times (n - m)}  & \, & \, \\
                                                 \, & 1 & \, \\
                                                 \, & \, & 0_{(m - 1) \times (m - 1)} \\
                                               \end{array}
                                             \right),
 \]
we thus obtain
\begin{equation} \label{eq3-1}
 \frac{1}{\alpha \beta c} \rho^{\alpha + 2} (c \rho^{- \alpha} +  d)^{ \frac{4 \beta}{n - 2} + 1} \psi^{ - \frac{n + 2}{n - 2}} W_{ij} [ \psi ]
=  A_m + B_{\alpha\beta}(\zeta) + \mathcal{O}(\rho),
\end{equation}
where $\zeta = \frac{ c \rho^{- \alpha}}{c \rho^{- \alpha} + d}$,
\[ A_m =  \left(
      \begin{array}{ccc}
        (n - m) I & \, & \, \\
        \, & n - m & \, \\
        \, & \, & (2 - m) I \\
      \end{array}
    \right),  \]
and
\[ \tiny{ B_{\alpha\beta}(\zeta) =
\left(
  \begin{array}{ccc}
    \big(  2 \alpha \beta \zeta + \alpha (1 - \zeta) + 2 - n \big)  I  & \, & \, \\
    \, & ( n - 1 )\alpha ( 1 - \zeta )   & \, \\
   \, & \, &  \big(  2 \alpha \beta \zeta + \alpha (1 - \zeta) + 2 - n \big)  I \\
  \end{array}
\right).}
\]

Since $v_m \notin \overline{\Gamma}_k$, we may choose $\alpha$ sufficiently small, and then an appropriate $\beta$ such that $\alpha \beta$ is slightly larger than $\frac{n}{2} - 1$ in order to make
\[ \lambda( A_m + B_{\alpha\beta}(\zeta) ) \notin \overline{\Gamma}_k \]
for all $0 < \zeta < 1$.

Then we can choose a sufficiently small $0< \rho_0 < 1$ such that on $0 < \rho < \rho_0$ and for all $0 < \zeta < 1$, the righthand side of \eqref{eq3-1} satisfies
\[ \lambda\big(  A_m + B_{\alpha\beta}(\zeta) + \mathcal{O}(\rho) \big) \notin  \overline{\Gamma}_k.  \]
It follows that
$\lambda \big( W_{ij}[\psi] \big) \notin \overline{\Gamma}_k$ on $\{ 0 < \rho < \rho_0 \}$  for all $c > 0$ and $d > 0$, which means that $\psi$ is a supersolution of \eqref{conformal2}.

Now we compare $u_{\Omega}$ and $\psi$. Choosing $d$ sufficiently large depending on $\rho_0$ such that on $\rho = \rho_0$ we have $\psi \geq d^{\beta} \geq u_{\Omega}$. By comparing $u_{\Omega}$ with the first function in Proposition \ref{subsolution} in the ball $B_{\rho(x)} (x)$ we can deduce that
\[ u_{\Omega}(x) \leq C(n, k) \rho^{1 - \frac{n}{2}} \leq c^{\beta} \rho^{- \alpha \beta} < \psi \]
in $0 < \rho \leq \delta$, where $\delta$ is a sufficiently small positive constant depending on $c$. By the maximum principle we have $u_{\Omega} \leq \psi$ in $\delta < \rho < \rho_0$. Letting $\delta \rightarrow 0$, we arrive at $u_{\Omega} \leq \psi$ in $0 < \rho < \rho_0$. Next letting $c \rightarrow 0$ we deduce that $u_{\Omega} \leq d^{\beta}$ in $0 < \rho < \rho_0$.

\vspace{4mm}

\section{More general equations}

\vspace{4mm}

By change of variable $u = e^{\frac{n - 2}{2} v}$, equation \eqref{eq4-1} is equivalent to
\begin{equation} \label{eq4-3}
\sigma_k \bigg( (n - 2) e^{- 2 v}  \mathcal{W} [ v ] \bigg) + \alpha(x) \sigma_{k - 1} \bigg( (n - 2) e^{- 2 v}  \mathcal{W} [ v ]  \bigg) = \alpha_0(x),
\end{equation}
where
\[  \mathcal{W}[v] = g^{-1} \bigg( \nabla^2 v - d v \otimes d v + \Big( \frac{\Delta v}{n - 2} + |\nabla v|^2 \Big) g - \frac{\text{Ric}_g}{n - 2} \bigg). \]
For convenience, we denote $\frac{\text{Ric}_g}{n - 2} = \mathcal{R}$.

Now we derive a priori estimates for $(k - 1)$-admissible solutions to equation \eqref{eq4-3} on a smooth compact manifold with boundary, that is, on $\overline{M} : = M \cup \partial M$.
We write \eqref{eq4-3} in the following form
\begin{equation} \label{eq4-4}
F \big( \mathcal{W} [ v ] \big) := \frac{\sigma_k \big( \mathcal{W} [ v ] \big)}{\sigma_{k - 1} \big( \mathcal{W} [ v ]  \big) } - \frac{\alpha_0(x) e^{2 k v}}{(n - 2)^k} \frac{1}{\sigma_{k - 1} \big( \mathcal{W} [ v ]  \big) } = - \frac{\alpha(x)}{n - 2} e^{2 v}.
\end{equation}
As shown in \cite{Guan-Zhang}, for $(k - 1)$-admissible $v$, namely $\lambda\big( \mathcal{W} [ v ] \big) \in \Gamma_{k - 1}$, equation \eqref{eq4-4} is elliptic and the operator $F \big( \mathcal{W} [ v ] \big)$
is concave with respect to $\{ \mathcal{W}_{ij} \}$.

\vspace{2mm}

\subsection{Gradient estimates}~

\vspace{2mm}

Let $v \in C^3(M) \cap C^1(\overline{M})$ be a $(k - 1)$-admissible solution of \eqref{eq4-3}.
We adopt the same test function as \cite{Guan08}: $\Phi = \zeta(x) w e^{\eta (v)}$, where $w = \frac{1}{2} |\nabla v|^2$, $\zeta$ and $\eta$ are functions to be chosen later. Assume that $\Phi$ attains its maximum at an interior point $x_0 \in M$. Choose a local orthonormal frame $e_1, \ldots, e_n$ about $x_0$. Then, at $x_0$, we have
\begin{equation} \label{eq4-5}
\frac{\nabla_i \zeta}{\zeta} + \frac{\nabla_i w}{w} + \eta' \nabla_i v = 0,
\end{equation}
\begin{equation} \label{eq4-6}
F^{ij} \Big( \frac{\nabla_{ij} \zeta}{\zeta} - \frac{\nabla_i \zeta \nabla_j \zeta}{\zeta^2} + \frac{\nabla_{ij} w}{w} - \frac{\nabla_i w \nabla_j w}{w^2} + \eta' \nabla_{ij} v + \eta'' \nabla_i v \nabla_j v \Big) \leq 0,
\end{equation}
where $F^{ij} = \frac{\partial F}{\partial \mathcal{W}_{ij}} \big( \mathcal{W} [v] \big)$.

By direct calculation,
\[ \nabla_i w = \nabla_{im} v \nabla_m v, \]
\[ \nabla_{ij} w = \nabla_{ijm} v \nabla_m v + \nabla_{im} v \nabla_{j m} v. \]
By \eqref{eq4-5}, we have
\begin{equation} \label{eq4-10}
F^{ij} \frac{\nabla_i w \nabla_j w}{w^2} \leq 3 F^{ij} \frac{\nabla_i \zeta \nabla_j \zeta}{\zeta^2} + \frac{3 \eta'^2}{2} F^{ij} \nabla_i v \nabla_j v.
\end{equation}
Hence \eqref{eq4-6} becomes
\begin{equation} \label{eq4-6-2}
\begin{aligned}
& F^{ij} \Big( \frac{\nabla_{ij} \zeta}{\zeta} - \frac{5 \nabla_i \zeta \nabla_j \zeta}{2 \zeta^2} \Big) +  \frac{1}{w} \Big( \delta_{lm} - \frac{\nabla_l v \nabla_m v}{2 w} \Big) F^{ij} \nabla_{il} v \nabla_{jm} v \\ & + \frac{F^{ij} \nabla_{ijm} v \nabla_m v}{w}
 + \eta' F^{ij} \nabla_{ij} v + \Big( \eta'' - \frac{3}{4} \eta'^2 \Big) F^{ij} \nabla_i v \nabla_j v \leq 0.
\end{aligned}
\end{equation}

Choose a smooth function $\zeta$ such that
\[ 0 \leq \zeta \leq 1, \quad |\nabla \zeta| \leq b_0 \sqrt{\zeta}, \quad |\nabla^2 \zeta| \leq b_0. \]
Then \eqref{eq4-6-2} reduces to
\begin{equation} \label{eq4-11}
\frac{F^{ij} \nabla_{ijm} v \nabla_m v}{w}
 + \eta' F^{ij} \nabla_{ij} v + \Big( \eta'' - \frac{3}{4} \eta'^2 \Big) F^{ij} \nabla_i v \nabla_j v \leq \frac{C}{\zeta} \sum F^{ii}.
\end{equation}
Similarly, we have
\begin{equation} \label{eq4-6-3}
\begin{aligned}
\frac{\Delta \nabla_{m} v \nabla_m v}{w}
 + \eta' \Delta v + \Big( \eta'' - \frac{3}{4} \eta'^2 \Big) | \nabla v |^2 \leq \frac{C}{\zeta}.
\end{aligned}
\end{equation}

Differentiating \eqref{eq4-4} yields
\begin{equation} \label{eq4-7}
F^{ij} \nabla_l \mathcal{W}_{ij} - \frac{\nabla_l \alpha_0 e^{2 k v} + \alpha_0 e^{2 k v} 2 k \nabla_l v}{(n - 2)^k} \frac{1}{\sigma_{k - 1} \big( \mathcal{W} [v] \big)} = - \frac{\nabla_l \alpha e^{2 v} + \alpha e^{2 v} 2 \nabla_l v}{n - 2},
\end{equation}
where
\[ \nabla_l \mathcal{W}_{ij} [v] = \nabla_{lij} v - \nabla_{li} v \nabla_j v - \nabla_i v \nabla_{lj} v + \Big( \frac{\nabla_l \Delta v}{n - 2} + 2 \nabla_l w \Big) \delta_{ij} - \nabla_l \mathcal{R}_{ij}. \]
Since
\[ \nabla_{ijl} v = \nabla_{lij} v + R_{lij}^m \nabla_m v,  \]
we have
\[ \nabla_l \Delta v = \Delta \nabla_l v - \sum\limits_m R_{lmm}^s \nabla_s v.  \]
Assume that $|\nabla v| \geq 1$. Combining \eqref{eq4-11}, \eqref{eq4-6-3} and \eqref{eq4-7} we obtain
\begin{equation} \label{eq4-9}
\begin{aligned}
& \eta' F^{ij} \nabla_{ij} v  + \eta' \frac{\Delta v}{n - 2} \sum F^{ii} \\
& + \Big( \eta'' - \frac{3}{4} \eta'^2 \Big) F^{ij} \nabla_i v \nabla_j v + \Big( \eta'' - \frac{3}{4} \eta'^2 \Big) \frac{| \nabla v |^2}{n - 2} \sum F^{ii} \\
& + \frac{e^{2 k v} \nabla_l \alpha_0 \nabla_l v}{(n - 2)^k \sigma_{k - 1} w} + \frac{ 4 k \alpha_0 e^{2 k v}}{(n - 2)^k \sigma_{k - 1}} - \frac{e^{2 v} \nabla_l \alpha \nabla_l v}{(n - 2) w} - \frac{4 \alpha e^{2 v}}{n - 2} \\
\leq  & \frac{C}{\zeta} \sum F^{ii}  - \frac{F^{ij} \nabla_{li} v \nabla_j v \nabla_l v}{w} - \frac{ F^{ij} \nabla_i v \nabla_{lj} v \nabla_l v}{w} + \frac{ 2 \nabla_l w \nabla_l v}{w} \sum F^{ii} .
\end{aligned}
\end{equation}

By \eqref{eq4-5},
\begin{equation} \label{eq4-12}
\begin{aligned}
& - \frac{F^{ij} \nabla_{li} v \nabla_j v \nabla_l v}{w} - \frac{ F^{ij} \nabla_i v \nabla_{lj} v \nabla_l v}{w} + \frac{ 2 \nabla_l w \nabla_l v}{w} \sum F^{ii} \\
= & F^{ij} \Big( \frac{\nabla_i \zeta}{\zeta} + \eta' \nabla_i v \Big) \nabla_j v + F^{ij} \nabla_i v \Big( \frac{\nabla_j \zeta}{\zeta} + \eta' \nabla_j v \Big)
\\ & - 2 \nabla_l v \Big( \frac{\nabla_l \zeta}{\zeta} + \eta' \nabla_l v \Big) \sum F^{ii} \\
\leq & \frac{C}{\sqrt{\zeta}} |\nabla v| \sum F^{ii} + 2 \eta' F^{ij} \nabla_i v \nabla_j v - 2 \eta' |\nabla v|^2 \sum F^{ii}.
\end{aligned}
\end{equation}
Also, we have
\begin{equation*}
F^{ij} \mathcal{W}_{ij} = \frac{\sigma_k}{\sigma_{k - 1}} + (k - 1) \frac{\alpha_0 e^{2k v}}{(n - 2)^k \sigma_{k - 1}},
\end{equation*}
and
\[ F^{ij} \mathcal{W}_{ij} = F^{ij} \nabla_{ij} v - F^{ij} \nabla_i v \nabla_j v +  \frac{\Delta v}{n - 2} \sum F^{ii} + |\nabla v|^2 \sum F^{ii} - F^{ij} \mathcal{R}_{ij}. \]
Therefore,
\begin{equation} \label{eq4-13}
\begin{aligned}
& \eta' F^{ij} \nabla_{ij} v  + \eta' \frac{\Delta v}{n - 2} \sum F^{ii} \\
\geq & \eta' \frac{\sigma_k}{\sigma_{k - 1}} + (k - 1) \frac{\alpha_0 e^{2 k v}}{(n - 2)^k \sigma_{k - 1}} \eta' \\
& + \eta' F^{ij} \nabla_i v \nabla_j v - \eta' |\nabla v|^2 \sum F^{ii} - C |\eta'| \sum F^{ii}.
\end{aligned}
\end{equation}
Taking \eqref{eq4-12} and \eqref{eq4-13} into \eqref{eq4-9}, and in view of \eqref{eq4-4}, we obtain
\begin{equation} \label{eq4-14}
\begin{aligned}
& \frac{e^{2 k v}}{(n - 2)^k \sigma_{k - 1}} \Big( k \eta' \alpha_0 + 4 k \alpha_0 + \frac{ \nabla_l \alpha_0 \nabla_l v}{w}\Big) \\
& + \Big( \eta'' - \frac{3}{4} \eta'^2 - \eta' \Big) F^{ij} \nabla_i v \nabla_j v + \Big( \eta'' - \frac{3}{4} \eta'^2 + (n - 2) \eta' \Big) \frac{| \nabla v |^2}{n - 2} \sum F^{ii} \\
\leq  & C \Big( \frac{1}{\zeta} + \frac{|\nabla v|}{\sqrt{\zeta}} + |\eta'| \Big) \sum F^{ii} + \Big( \eta' \alpha + 4 \alpha + \frac{\nabla_l \alpha \nabla_l v}{w} \Big) \frac{e^{2 v}}{n - 2}.
\end{aligned}
\end{equation}

Choose
\[  \eta (v) = \Big( \frac{3}{2} + v - \inf_{\{ \zeta > 0 \} } v \Big)^{- N}, \]
where $N \geq 1$ is sufficiently large such that
\[\begin{aligned}
& \eta'' - \frac{3}{4} \eta'^2 - \eta' \\
\geq & \eta'' - \frac{3}{4} \eta'^2 + (n - 2) \eta'   \\
\geq & \frac{N^2}{2} \Big( \frac{3}{2} + v - \inf_{\{ \zeta > 0 \} } v \Big)^{- N - 2}.
\end{aligned} \]
By Newton-Maclaurin inequality
\[ \sigma_k \sigma_{k - 2} \leq \frac{(n - k + 1)(k - 1)}{(n - k + 2) k} \sigma_{k - 1}^2, \]
we know that
\[\begin{aligned}
\sum F^{ii} = & (n - k + 1) - \frac{(n - k + 2) \sigma_k \sigma_{k - 2}}{\sigma_{k - 1}^2} + \frac{(n - k + 2) \alpha_0 e^{2 k v} \sigma_{k - 2}}{(n - 2)^k \sigma_{k - 1}^2} \\
\geq & \frac{n - k + 1}{k}.
\end{aligned} \]
Therefore,
\[
\bigg( \eta' \alpha + 4 \alpha + \frac{\nabla_l \alpha \nabla_l v}{w} \bigg) \frac{e^{2 v}}{n - 2} \leq C \bigg( |\eta'| + 1 + \frac{1}{|\nabla v|} \bigg) \sum F^{ii}.
 \]
Also, we know that
\[  k \eta' \alpha_0 + 4 k \alpha_0 + \frac{ \nabla_l \alpha_0 \nabla_l v}{w} \geq \alpha_0 \bigg( - k N \Big(\frac{2}{3} \Big)^{N + 1} + 4 k - \frac{C}{\alpha_0 |\nabla v|} \bigg) > 0 \]
for $N$ sufficiently large and $|\nabla v| \geq \frac{C}{2 k \inf \alpha_0}$.
Hence, \eqref{eq4-14} reduces to
\[ |\nabla v|^2 \leq C \Big( \frac{1}{\zeta} + \frac{|\nabla v|}{\sqrt{\zeta}} \Big). \]
Consequently,
\begin{equation} \label{eq4-33}
\sqrt{\zeta} |\nabla v| \leq C.
\end{equation}

Taking $\zeta$ to be a standard cutoff function in a geodesic ball $B_r$ of radius $r > 0$ satisfying
\[ \vert \nabla \zeta \vert \leq \frac{C}{r}, \]
by \eqref{eq4-33} we obtain

\begin{thm} \label{interior gradient estimate}
If $v \in C^3(B_r)$ is a $(k - 1)$-admissible solution of \eqref{eq4-3} in a geodesic ball $B_r \subset M$ of radius $r > 0$, then
\[ \sup\limits_{B_{\frac{r}{2}}} \vert \nabla v \vert \leq C, \]
where $C$ depends on $r^{- 1}$, $n$, $k$, $\Vert v \Vert_{C^0 (B_r)}$, $\Vert \alpha \Vert_{C^1(B_r)}$, $\Vert \alpha_0 \Vert_{C^1{(B_r)}}$ and $\inf\limits_{B_r} \alpha_0$.
\end{thm}

Let $\zeta \equiv 1$. By \eqref{eq4-33} we have

\begin{thm} \label{global gradient estimate}
Let $v \in C^3(M) \cap C^1 (\overline{M})$ be a $(k - 1)$-admissible solution of \eqref{eq4-3} in $\overline{M}$. Then
\[ \max\limits_{\overline{M}} \vert \nabla v \vert \leq C, \]
where $C$ depends on $n$, $k$, $\Vert v \Vert_{C^0 (\overline{M})}$, $\max\limits_{\partial M} \vert \nabla v \vert$, $\Vert \alpha \Vert_{C^1(M)}$, $\Vert \alpha_0 \Vert_{C^1{(M)}}$ and $\inf\limits_{M} \alpha_0$.
\end{thm}

\vspace{2mm}

\subsection{Interior and global estimates for second derivatives}~

\vspace{2mm}

Let $v \in C^4(M) \cap C^2(\overline{M})$ be a $(k - 1)$-admissible solution of \eqref{eq4-3}. We adopt the idea of Guan \cite{Guan08} to derive interior and global estimates for second derivatives.
Consider the test function
\[ \Psi = \zeta e^{\eta(w)} \Big( \nabla_{\xi \xi} v - (\nabla_{\xi} v)^2 - \mathcal{R}_{\xi\xi} \Big), \]
where $\xi$ is a unit tangent vector to $\overline{M}$ at $x$, $w = \frac{1}{2} |\nabla v|^2$, $\zeta$ and $\eta$ are functions to be chosen later, with $\zeta$ smooth and satisfying
\[ 0 \leq \zeta \leq 1, \quad |\nabla \zeta| \leq b_0, \quad |\nabla^2 \zeta| \leq b_0. \]
Assume that $\Psi$ is attained at some interior point $x_0 \in M$ and for some unit vector $\xi \in T_{x_0} \overline{M}$. Choose a smooth local orthonormal frame $e_1, \ldots, e_n$ about $x_0$ such that $e_1(x_0) = \xi$ and $\big\{ \mathcal{W}_{ij}[v] (x_0) \big\}$ is diagonal. Let $G =  \nabla_{11} v - (\nabla_{1} v)^2 - \mathcal{R}_{11}$ and assume that $G (x_0) \geq 1$.
At $x_0$, where the function $ \zeta e^{\eta(w)} G$, which is locally defined near $x_0$, attains its maximum,
\begin{equation} \label{eq4-15}
\frac{\nabla_{i} G}{G} + \nabla_i \eta + \frac{\nabla_i \zeta}{\zeta} = 0,
\end{equation}
and
\begin{equation} \label{eq4-16}
F^{ii} \bigg( \frac{\nabla_{ii} G}{G} - \Big( \frac{ \nabla_{i} G }{G} \Big)^2 + \nabla_{ii} \eta + \frac{\nabla_{ii} \zeta}{\zeta} - \Big( \frac{\nabla_i \zeta}{\zeta} \Big)^2 \bigg) \leq 0.
\end{equation}

By \eqref{eq4-15} and Cauchy-Schwartz inequality,
\[ \bigg(\frac{\nabla_{i} G}{G} \bigg)^2 \leq (1 + \epsilon) (\nabla_i \eta)^2 + \Big( 1 + \frac{1}{\epsilon} \Big) \Big( \frac{\nabla_i \zeta}{\zeta} \Big)^2, \quad \epsilon > 0. \]
Therefore, \eqref{eq4-16} reduces to
\begin{equation} \label{eq4-17}
F^{ii} \bigg( \frac{\nabla_{ii} G}{G} + \nabla_{ii} \eta - (1 + \epsilon) (\nabla_i \eta)^2  \bigg) \leq \frac{C}{\zeta^2} \sum F^{ii}.
\end{equation}
Similarly,
\begin{equation} \label{eq4-18}
\frac{\Delta G}{G} + \Delta \eta - (1 + \epsilon) |\nabla \eta|^2 \leq \frac{C}{\zeta^2}.
\end{equation}
Next, we compute
\begin{equation} \label{eq4-35}
\nabla_i G = \nabla_{i11} v - 2 \nabla_1 v \nabla_{i1} v - \nabla_i \mathcal{R}_{11},
\end{equation}
\begin{equation} \label{eq4-36}
\begin{aligned}
\nabla_{ii} G = & \nabla_{ii11} v - 2 \nabla_1 v \nabla_{ii1} v - 2 (\nabla_{i1} v)^2 - \nabla_{ii} \mathcal{R}_{11} \\
\geq & \nabla_{11ii} v - 2 \nabla_1 v \nabla_{1ii} v - 2 ( \nabla_{i1} v )^2 - C \bigg( 1 + \sum\limits_{j l} |\nabla_{jl} v| \bigg),
\end{aligned}
\end{equation}
\begin{equation} \label{eq4-37}
\Delta G \geq \nabla_{11} \Delta v - 2 \nabla_1 v \nabla_{1} \Delta v - 2 \sum\limits_i (\nabla_{i1} v)^2 - C \bigg( 1 + \sum\limits_{j l} |\nabla_{jl} v| \bigg),
\end{equation}
\begin{equation} \label{eq4-38}
\begin{aligned}
& F^{ii} \nabla_{1} \bigg( \nabla_{ii} v + \frac{\Delta v}{n - 2} \bigg) =  F^{ii} \nabla_{1} \Big( \mathcal{W}_{ii} + (\nabla_i v)^2 - |\nabla v|^2 + \mathcal{R}_{ii} \Big)   \\
= & F^{ii} \nabla_1 \mathcal{W}_{ii} + 2 F^{ii} \nabla_i v \nabla_{1i} v - 2 \sum\limits_k \nabla_k v \nabla_{1k} v \sum F^{ii} + F^{ii} \nabla_1 \mathcal{R}_{ii},
\end{aligned}
\end{equation}
\begin{equation} \label{eq4-19}
\begin{aligned}
& F^{ii} \nabla_{11} \bigg( \nabla_{ii} v + \frac{\Delta v}{n - 2} \bigg) =  F^{ii} \nabla_{11} \Big( \mathcal{W}_{ii} + ( \nabla_i v )^2 - |\nabla v|^2 + \mathcal{R}_{ii} \Big)   \\
\geq &  F^{ii} \nabla_{11} \mathcal{W}_{ii} + 2 F^{ii} \nabla_i v \nabla_{11i} v + 2 F^{ii} (\nabla_{1i} v )^2 \\
& - 2 \sum\limits_k \nabla_{11k} v \nabla_k v \sum\limits_i F^{ii} - 2 \sum\limits_k ( \nabla_{1k} v )^2 \sum\limits_i F^{ii} - C \sum F^{ii}.
\end{aligned}
\end{equation}
By \eqref{eq4-15}, \eqref{eq4-35} and the formula for interchanging order of covariant derivatives,
\begin{equation} \label{eq4-39}
\begin{aligned}
& 2 F^{ii} \nabla_i v \nabla_{11i} v  - 2 \sum\limits_k \nabla_{11k} v \nabla_k v \sum\limits_i F^{ii} \\
\geq & 4 \nabla_1 v F^{ii} \nabla_{i1} v \nabla_i v - 4 \nabla_1 v \sum\limits_k \nabla_{k1} v \nabla_k v \sum F^{ii} \\
& - C \bigg( |\nabla \eta| + \frac{1}{\zeta} \bigg) G \sum F^{ii}.
\end{aligned}
\end{equation}
Combining \eqref{eq4-36}--\eqref{eq4-39} yields,
\begin{equation} \label{eq4-20}
\begin{aligned}
& F^{ii} \bigg( \nabla_{ii} G + \frac{\Delta G}{n - 2} \bigg) \\
\geq & F^{ii} \nabla_{11} \mathcal{W}_{ii} - 2 \nabla_1 v F^{ii} \nabla_1 \mathcal{W}_{ii} - \Big( 2 + \frac{2}{n - 2} \Big) \sum\limits_k (\nabla_{1k} v )^2 \sum F^{ii} \\
& - C \bigg( |\nabla \eta| G + \frac{G}{\zeta} + \sum\limits_{j, l} |\nabla_{jl} v| \bigg) \sum F^{ii}.
\end{aligned}
\end{equation}

Differentiating equation \eqref{eq4-4} twice and by the concavity of $\frac{\sigma_k}{\sigma_{k - 1}}$,
\begin{equation} \label{eq4-21}
\begin{aligned}
& F^{ii} \nabla_{11} \mathcal{W}_{ii}  + \frac{\alpha_0 e^{2 k v}}{(n - 2)^k \sigma_{k - 1}} \bigg( \frac{\sigma_{k - 1}^{ij, rs}}{\sigma_{k - 1}} - \frac{2 \sigma_{k - 1}^{ij} \sigma_{k - 1}^{rs}}{\sigma_{k - 1}^2} \bigg) \nabla_1 \mathcal{W}_{ij} \nabla_1 \mathcal{W}_{rs} \\
& + \frac{\big( 2 \nabla_1 \alpha_0 e^{2 k v} + 4 k \alpha_0 e^{2 k v} \nabla_1 v \big) \sigma_{k - 1}^{ii} \nabla_1 \mathcal{W}_{ii} }{(n - 2)^k \sigma_{k - 1}^2}  \\
& - \frac{\nabla_1 \alpha_0 e^{2kv} 4 k \nabla_1 v + \alpha_0 e^{2k v} 4 k^2 (\nabla_1 v)^2 + \alpha_0 e^{2 k v} 2 k \nabla_{11} v + \nabla_{11} \alpha_0 e^{2k v}}{(n - 2)^k \sigma_{k - 1}} \\
\geq & - \frac{\nabla_{11} \alpha e^{2v} + 4 \nabla_1 \alpha e^{2 v} \nabla_1 v + 4 \alpha e^{2 v} (\nabla_1 v)^2 + 2 \alpha e^{2v} \nabla_{11} v}{n - 2}.
\end{aligned}
\end{equation}
By the concavity of $\sigma_{k - 1}^{\frac{1}{k - 1}}$,
\[ \bigg( \frac{\sigma_{k - 1}^{ij, rs}}{\sigma_{k - 1}} + \Big( \frac{1}{k - 1} - 1 \Big) \frac{\sigma_{k - 1}^{ij} \sigma_{k - 1}^{rs}}{\sigma_{k - 1}^2} \bigg) \nabla_1 \mathcal{W}_{ij} \nabla_1 \mathcal{W}_{rs} \leq 0. \]
Consequently,
\[ \begin{aligned}
& \frac{\alpha_0 e^{2 k v}}{(n - 2)^k \sigma_{k - 1}} \bigg( \frac{\sigma_{k - 1}^{ij, rs}}{\sigma_{k - 1}} - \frac{2 \sigma_{k - 1}^{ij} \sigma_{k - 1}^{rs}}{\sigma_{k - 1}^2} \bigg) \nabla_1 \mathcal{W}_{ij} \nabla_1 \mathcal{W}_{rs} \\
& + \frac{\big( 2 \nabla_1 \alpha_0 e^{2 k v} + 4 k \alpha_0 e^{2 k v} \nabla_1 v \big) \sigma_{k - 1}^{ii} \nabla_1 \mathcal{W}_{ii} }{(n - 2)^k \sigma_{k - 1}^2} \\
\leq & - \frac{\alpha_0 e^{2 k v}}{(n - 2)^k \sigma_{k - 1}}  \frac{k}{k - 1} \frac{\sigma_{k - 1}^{ij} \sigma_{k - 1}^{rs}}{\sigma_{k - 1}^2}  \nabla_1 \mathcal{W}_{ij} \nabla_1 \mathcal{W}_{rs} \\
& + \frac{\big( 2 \nabla_1 \alpha_0 e^{2 k v} + 4 k \alpha_0 e^{2 k v} \nabla_1 v \big) \sigma_{k - 1}^{ii} \nabla_1 \mathcal{W}_{ii} }{(n - 2)^k \sigma_{k - 1}^2} \\
\leq & \frac{(k - 1) e^{2 k v} ( \nabla_1 \alpha_0 + 2 k \alpha_0 \nabla_1 v )^2}{k (n - 2)^k \alpha_0 \sigma_{k - 1}}.
\end{aligned} \]
By \eqref{eq4-4} and Newton-Maclaurin inequality,
\[\begin{aligned}
\frac{\alpha_0}{\sigma_{k - 1}} =  \frac{(n - 2)^k}{e^{2 k v}} \bigg( \frac{\sigma_k}{\sigma_{k - 1}} + \frac{\alpha e^{2 v}}{n - 2} \bigg)
\leq  C \Big( \sigma_{k - 1}^{\frac{1}{k - 1}} + 1 \Big) \leq C (\sigma_1 + 1) \leq C G.
\end{aligned} \]
Also we have
\[ \vert \nabla_{ij} v \vert \leq C G, \quad \forall \, i, j.  \]
Therefore, \eqref{eq4-21} reduces to
\[ F^{ii} \nabla_{11} \mathcal{W}_{ii} \geq - C G^2 .  \]
Combining this inequality with \eqref{eq4-20}, \eqref{eq4-17}, \eqref{eq4-18} and in view of \eqref{eq4-7} yields,
\begin{equation} \label{eq4-22}
\begin{aligned}
& F^{ii} \bigg( \nabla_{ii} \eta - (1 + \epsilon) ( \nabla_i \eta )^2 \bigg) + \frac{1}{n - 2} \bigg( \Delta \eta - (1 + \epsilon) |\nabla \eta|^2 \bigg) \sum F^{ii} \\
\leq & C \bigg( |\nabla \eta| + G + \frac{1}{\zeta^2} \bigg) \sum F^{ii}.
\end{aligned}
\end{equation}

Let
\[ \eta(w) = \bigg( 1 - \frac{3 w}{4 M} \bigg)^{- 1/2},  \]
where
\[ w = \frac{|\nabla v|^2}{2}, \quad M = \sup\limits_{\{\zeta > 0 \}} w. \]
Choosing $\epsilon = \frac{1}{2}$, we have
\[ \eta'' - (1 + \epsilon) \eta'^2  = \frac{9}{64 M^2} \bigg( 1 - \frac{3 w}{4 M} \bigg)^{- \frac{5}{2}} \Bigg( 3 - (1 + \epsilon) \bigg( 1 - \frac{3 w}{4 M} \bigg)^{- \frac{1}{2}}  \Bigg) \geq 0. \]
Therefore,
\begin{equation*}
\begin{aligned}
& F^{ii} \bigg( \nabla_{ii} \eta - (1 + \epsilon) ( \nabla_i \eta )^2 \bigg) + \frac{1}{n - 2} \bigg( \Delta \eta - (1 + \epsilon) |\nabla \eta|^2 \bigg) \sum F^{ii} \\
\geq & \eta' F^{ii} \big( \nabla_{iil} v \nabla_l v + ( \nabla_{il} v )^2 \big) + \frac{1}{n - 2} \eta' \bigg( \Delta \nabla_l v \nabla_l v + \sum\limits_{jl} ( \nabla_{jl} v )^2 \bigg) \sum F^{ii}.
\end{aligned}
\end{equation*}
By \eqref{eq4-7} and interchanging order of covariant derivatives,
\[  F^{ii} \nabla_{iil} v \nabla_l v + \frac{1}{n - 2} \Delta \nabla_l v \nabla_l v \sum F^{ii} \geq - C G \sum F^{ii}. \]
We note that
\[ \frac{3}{8 M} \leq \eta' \leq \frac{3}{M}, \quad |\nabla \eta| \leq C G,  \]
and
\[ \sum\limits_{j l} (\nabla_{jl} v)^2 \geq (\nabla_{11} v)^2 \geq G^2 - C G. \]
Hence \eqref{eq4-22} reduces to
\[  G^2 \leq C \Big( G + \frac{1}{\zeta^2} \Big). \]
Consequently,
\begin{equation} \label{eq4-34}
\zeta(x_0) G(x_0) \leq C.
\end{equation}

Finally, choosing $\zeta$ to be an appropriate cutoff function with support in $B_r$, we obtain the second order interior estimate for equation \eqref{eq4-3}.

\begin{thm} \label{second order interior estimate}
Let $v \in C^4(B_r)$ be a $(k - 1)$-admissible solution of \eqref{eq4-3} in a geodesic ball $B_r \subset M$ of radius $r > 0$. Then
\[ \sup\limits_{B_{\frac{r}{2}}} \vert \nabla^2 v \vert \leq C, \]
where $C$ depends on $r^{- 1}$, $n$, $k$, $\Vert v \Vert_{C^1(B_r)}$, $\Vert \alpha \Vert_{C^2(B_r)}$, $\Vert \alpha_0 \Vert_{C^2{(B_r)}}$ and $\inf\limits_{B_r} \alpha_0$.
\end{thm}

If we choose $\zeta \equiv 1$ in \eqref{eq4-34}, we obtain the global second order estimate.

\begin{thm} \label{global second order estimate}
Let $v \in C^4(M) \cap C^2 (\overline{M})$ be a $(k - 1)$-admissible solution of \eqref{eq4-3}. Then
\[ \max\limits_{\overline{M}} \vert \nabla^2 v \vert \leq C,  \]
where $C$ depends on $n$, $k$, $\Vert v \Vert_{C^1 (\overline{M})}$, $\max\limits_{\partial M} \vert \nabla^2 v \vert$, $\Vert \alpha \Vert_{C^2(M)}$, $\Vert \alpha_0 \Vert_{C^2{(M)}}$ and $\inf\limits_{M} \alpha_0$.
\end{thm}

\vspace{2mm}

\subsection{Boundary estimates for second derivatives}

\vspace{2mm}

\begin{thm} \label{second order boundary estimate}
Let $v \in C^3(\overline{M})$ be a $(k - 1)$-admissible solution of \eqref{eq4-3} satisfying
\[ v = \varphi \quad \text{  on  } \partial M,   \]
where $\varphi \in C^{\infty} (\overline{M})$. Then we have the estimate
\[ |\nabla^2 v| \leq C \quad \text{  on  }  \partial M, \]
where $C > 0$ is a constant depending on $n$, $k$, $\Vert v \Vert_{C^1(\overline{M})}$, $\Vert \varphi \Vert_{C^3(\overline{M})}$, $\Vert \alpha \Vert_{C^1(\overline{M})}$, $\Vert \alpha_0 \Vert_{C^1(\overline{M})}$ and $\min\limits_{\overline{M}} \alpha_0$.
\end{thm}
\begin{proof}
Let $x_0$ be an arbitrary point on $\partial M$. Choose a smooth local orthonormal frame $e_1, \ldots, e_n$ around $x_0$ such that $e_n$ is the interior unit normal vector field to $\partial M$ along $\partial M$.  We obtain the pure tangential second derivative bound immediately
\begin{equation} \label{eq4-28}
\big\vert \nabla_{s t} v (x_0) \big\vert \leq C, \quad \forall \, s, t < n.
\end{equation}

For the tangential-normal second derivative estimate
\begin{equation} \label{eq4-27}
\big\vert \nabla_{s n} v (x_0) \big\vert \leq C, \quad \forall \, s < n,
\end{equation}
we use the function
\[ \Theta = \beta \bigg( \rho - \frac{N}{2} \rho^2 \bigg), \]
where $\beta$ and $N$ are positive constants, and $\rho(x)$ is the distance from $x$ to $\partial M$.
Set
\[ M_{\delta} = \{ x \in M \, \vert \, \rho(x) < \delta \}. \]
Choose $\delta > 0$ sufficiently small such that $\rho(x)$ is smooth in $M_{\delta}$, on which, we have
\[ \vert \nabla \rho \vert = 1, \quad \vert \nabla^2 \rho \vert \leq C. \]
We consider the linearized operator $\mathcal{L}$ which is locally defined by
\[ \mathcal{L} \eta = F^{ij} \big( \nabla_{ij} \eta - 2 \nabla_i v \nabla_j \eta \big) + \bigg( \frac{\Delta \eta}{n - 2} + 2 \nabla_m v \nabla_m \eta \bigg) \sum F^{ii}.  \]
It can be seen that
\[ \begin{aligned}
\mathcal{L} \Theta \leq \beta \bigg( C (1 + N \delta) - \frac{N}{n - 2} \bigg) \sum F^{ii}.
\end{aligned} \]
Choosing
\[ N \geq (n - 2) (2 C + 1), \quad \text{ and } \quad \delta \leq \frac{1}{N}, \]
we have
\begin{equation} \label{eq4-23}
\mathcal{L} \Theta \leq - \beta \sum F^{ii}  \quad \text{  in  } M_{\delta}
\end{equation}
and
\begin{equation} \label{eq4-24}
\Theta \geq \frac{\beta}{2} \rho \quad \text{  in  } M_{\delta}.
\end{equation}

We note that
\[ \nabla_{ij} (\nabla_k v) = \nabla_{kij} v + \Gamma_{ik}^l \nabla_{jl} v + \Gamma_{jk}^l \nabla_{il} v + \nabla_k \Gamma_{ij}^l \nabla_l v. \]
Therefore, for $t < n$,
\[\begin{aligned}
& \mathcal{L} \big( \nabla_{t} (v - \varphi) \big) \\
= & F^{ij} \Big( \nabla_{t ij}(v - \varphi) + \Gamma_{it}^l \nabla_{jl} (v - \varphi) + \Gamma_{jt}^l \nabla_{il}(v - \varphi) + \nabla_{t} \Gamma_{ij}^l \nabla_l (v - \varphi) \Big) \\
& - 2 F^{ij} \nabla_i v \Big( \nabla_{jt}(v - \varphi) + \Gamma_{jt}^l \nabla_l (v - \varphi)  \Big) \\
& + \Bigg( \frac{\nabla_{t} \Delta (v - \varphi) + 2 \sum_j \Gamma_{jt}^l \nabla_{jl} (v - \varphi) + \sum_j \nabla_{t} \Gamma_{jj}^l \nabla_l (v - \varphi)}{n - 2} \\
& + 2 \nabla_m v \Big( \nabla_{m t}(v - \varphi) + \Gamma_{mt}^l \nabla_l (v - \varphi) \Big) \Bigg) \sum F^{ii}.
\end{aligned} \]
By \eqref{eq4-7}, we have
\[ \begin{aligned}
& F^{ij} \Big( \nabla_{t ij} v - 2 \nabla_i v \nabla_{tj} v \Big) + \Bigg( \frac{\nabla_t \Delta v}{n - 2} + \nabla_t |\nabla v|^2 \Bigg) \sum F^{ii} \\
\leq  & C \sum F^{ii} + \frac{C \alpha_0}{\sigma_{k - 1}}
\leq  C \sum F^{ii} + C \big( \sigma_1 + 1 \big)
\leq  C \sum F^{ii} + C \sigma_1.
\end{aligned} \]
Consequently, for $t < n$,
\begin{equation} \label{eq4-25}
\begin{aligned}
& \mathcal{L} \big( \nabla_{t} (v - \varphi) \big) \\
\leq & C \sum F^{ii} + C \sigma_1 + C \sqrt{\sum_{jl} (\nabla_{jl} v)^2} \sum F^{ii} \\
\leq & C \Bigg( 1 + \sqrt{\sum_{jl} (\nabla_{jl} v)^2} \Bigg) \sum F^{ii}.
\end{aligned}
\end{equation}
Similarly, we may verify that
\begin{equation} \label{eq4-25-1}
\begin{aligned}
 \mathcal{L} \big( - \nabla_{t} (v - \varphi) \big)
\leq  C \Bigg( 1 + \sqrt{\sum_{jl} (\nabla_{jl} v)^2} \Bigg) \sum F^{ii}.
\end{aligned}
\end{equation}
We also have
\begin{equation} \label{eq4-26}
\begin{aligned}
& \mathcal{L} \Bigg( \sum\limits_{s < n} \big| \nabla_s (v - \varphi) \big|^2 \Bigg)
=  \sum\limits_{s < n} 2 \nabla_s (v - \varphi) \mathcal{L} \big( \nabla_s (v - \varphi) \big) \\
& + \sum\limits_{s < n} 2 F^{ij} \nabla_i \big( \nabla_s (v - \varphi) \big) \nabla_j \big( \nabla_s (v - \varphi) \big)
 + \sum\limits_{s < n} \frac{2 \Big| \nabla \big( \nabla_s (v - \varphi) \big) \Big|^2}{n - 2} \sum F^{ii} \\
\geq & - C \Bigg( 1 + \sqrt{\sum_{jl} (\nabla_{jl} v)^2} \Bigg) \sum F^{ii} + 0
 + \frac{\sum\limits_l \sum\limits_{s < n} (\nabla_{ls} v)^2}{n - 2} \sum F^{ii} - C \sum F^{ii}.
\end{aligned}
\end{equation}

We observe that if $\nabla_{nn} v \leq 0$, by the fact that
\[ 0 \leq \sigma_1 \big( \mathcal{W}[v] \big) \leq \frac{2(n - 1)}{n - 2} \Delta v + C, \]
we have
\[ 0 \geq \nabla_{nn} v \geq - \sum\limits_{s < n} \nabla_{ss} v - C.   \]

If $\nabla_{nn} v \geq 0$, by the concavity of $F$ with respect to $\{\mathcal{W}_{ij}\}$, that is,
\[ \sum\limits_{ij} F^{ij} \big( \delta_{ij} - \mathcal{W}_{ij} \big) \geq F \big( I \big) - F \big( \mathcal{W}[v] \big) \geq - C, \]
we have
\[ \begin{aligned}
& \sum F^{ii} + C \geq \sum\limits_{ij} F^{ij} \mathcal{W}_{ij} \\
\geq & \sum\limits_{ij} F^{ij}  \bigg( \nabla_{ij} v + \frac{\Delta v}{n - 2} \delta_{ij} \bigg) - C \sum F^{ii} \\
= & \sum\limits_{i} \sum\limits_{s < n} F^{is} \nabla_{is} v + \sum\limits_{s < n} F^{s n} \nabla_{s n} v + F^{nn} \nabla_{nn} v + \frac{\Delta v}{n - 2} \sum F^{ii} - C \sum F^{ii} \\
\geq & - C \sqrt{\sum_{l} \sum_{s < n} (\nabla_{ls} v)^2} \sum F^{ii} + \frac{\Delta v}{n - 2} \sum F^{ii} - C \sum F^{ii},
\end{aligned} \]
which implies that
\[ 0 \leq \nabla_{nn} v \leq C + C \sqrt{\sum_{l} \sum_{s < n} (\nabla_{ls} v)^2}. \]

By the above observation, \eqref{eq4-23}, \eqref{eq4-25}, \eqref{eq4-25-1} and \eqref{eq4-26}, we have
\begin{equation*}
\begin{aligned}
& \mathcal{L} \Bigg( \Theta - \sum\limits_{s < n} \big| \nabla_s (v - \varphi) \big|^2 \pm \nabla_t (v - \varphi) \Bigg) \\
\leq & - \beta \sum F^{ii} + C \sqrt{\sum\limits_l \sum\limits_{s < n} (\nabla_{ls} v)^2} \sum F^{ii}
- \frac{\sum\limits_l \sum\limits_{s < n} (\nabla_{ls} v)^2}{n - 2} \sum F^{ii} + C \sum F^{ii}.
\end{aligned}
\end{equation*}

If
\[ \sqrt{\sum\limits_l \sum\limits_{s < n} (\nabla_{ls} v)^2} \leq C (n - 2), \]
then we obtain
\[  \mathcal{L} \Bigg( \Theta - \sum\limits_{s < n} \big| \nabla_s (v - \varphi) \big|^2 \pm \nabla_t (v - \varphi) \Bigg)
\leq  - \beta \sum F^{ii} + C^2 (n - 2) \sum F^{ii} + C \sum F^{ii}. \]

If
\[ \sqrt{\sum\limits_l \sum\limits_{s < n} (\nabla_{ls} v)^2} \geq C (n - 2), \]
then we obtain
\[  \mathcal{L} \Bigg( \Theta - \sum\limits_{s < n} \big| \nabla_s (v - \varphi) \big|^2 \pm \nabla_t (v - \varphi) \Bigg)
\leq - \beta \sum F^{ii} + C \sum F^{ii}. \]
Thus, choosing $\beta$ sufficiently large, we arrive at
\[  \mathcal{L} \Bigg( \Theta - \sum\limits_{s < n} \big| \nabla_s (v - \varphi) \big|^2 \pm \nabla_t (v - \varphi) \Bigg) \leq 0. \]
Also,
\[  \Theta - \sum\limits_{s < n} \big| \nabla_s (v - \varphi) \big|^2 \pm \nabla_t (v - \varphi) = 0 \quad \text{  on  } \partial M,  \]
and by \eqref{eq4-24},  we may choose $\beta$ further large depending on $\delta$ such that
\[  \Theta - \sum\limits_{s < n} \big| \nabla_s (v - \varphi) \big|^2 \pm \nabla_t (v - \varphi) \geq \frac{\beta}{2} \delta - C \geq 0 \quad \text{  on  }  \{ x \in M \,| \, \rho(x) = \delta \}.  \]
By the maximum principle,
\[  \Theta - \sum\limits_{s < n} \big| \nabla_s (v - \varphi) \big|^2 \pm \nabla_t (v - \varphi) \geq 0 \quad \text{  in  }  M_{\delta}.  \]
Hence we obtain \eqref{eq4-27}.

Next, we derive the double normal second derivative estimate
\begin{equation} \label{eq4-29}
\nabla_{n n} v (x_0)  \leq C.
\end{equation}
By \eqref{eq4-3}, that is,
\[ \sigma_k \big( \mathcal{W} [ v ] \big) + \frac{\alpha e^{2 v}}{n - 2} \sigma_{k - 1} \big( \mathcal{W} [ v ]  \big) = \frac{\alpha_0 e^{2 k v}}{(n - 2)^k}, \]
if $\nabla_{n n} v (x_0)$ is sufficiently large, in view of \eqref{eq4-28} and \eqref{eq4-27}, we have at $x_0$,
\[ \sigma_k \bigg( \frac{\Delta v}{n - 2} \delta_{ij} - C \delta_{ij} \bigg) - C \sigma_{k - 1} \bigg( \frac{\Delta v}{n - 2} \delta_{ij} + \nabla_{nn} v \delta_{ij} + C \delta_{ij} \bigg) \leq C, \]
which further implies that
\[ \bigg( \frac{1}{n - 2} \nabla_{nn} v - C \bigg)^k \sigma_k (I) - C \bigg( \frac{n - 1}{n - 2} \nabla_{nn} v + C \bigg)^{k - 1} \sigma_{k - 1} (I) \leq C. \]
Hence we obtain \eqref{eq4-29}.

\end{proof}

\begin{rem}
When deriving the estimates, we can also directly obtain
\[ \frac{1}{\sigma_{k - 1} \big( \mathcal{W} [ v ] \big) } \leq C \]
by \eqref{eq4-3}. In fact, we have
\[ \begin{aligned}
\inf\alpha_0 \leq & \alpha_0(x) = \sigma_k \bigg( (n - 2) e^{- 2 v}  \mathcal{W} [ v ] \bigg) + \alpha(x) \sigma_{k - 1} \bigg( (n - 2) e^{- 2 v}  \mathcal{W} [ v ]  \bigg) \\
\leq &  \sigma_{k - 1}^{\frac{k}{k - 1}} \bigg( (n - 2) e^{- 2 v}  \mathcal{W} [ v ] \bigg) + \sup\alpha \sigma_{k - 1} \bigg( (n - 2) e^{- 2 v}  \mathcal{W} [ v ]  \bigg).
\end{aligned} \]
\end{rem}

\vspace{4mm}

Next, we shall use continuity method and degree theory to prove the existence of a smooth $(k - 1)$-admissible solution to the Dirichlet problem
\begin{equation} \label{eq4-30}
\left\{\begin{aligned}
F \big( \mathcal{W} [ v ] \big) = & - \frac{\alpha(x)}{n - 2} e^{2 v} \quad &\text{  in  } \Omega, \\
v = & \varphi \quad &\text{  on  } \partial\Omega,
\end{aligned} \right.
\end{equation}
where $\Omega$ is a bounded domain in $\mathbb{R}^n$ with smooth compact boundary which are composed of closed hypersurfaces.

\vspace{2mm}

\subsection{Existence of subsolutions}

\vspace{2mm}

In this subsection, we construct a subsolution synthesizing the ideas of Guan \cite{Guan08} and Guan \cite{GuanP02}.
We note that there exist sufficiently large $R > r > 0$ such that $\Omega \subset B_r (0)$, and
\[ w (x) = - \ln \big( R^2 - |x|^2 \big) - C \]
satisfies
\[ \mathcal{W} [ w ] \geq \nabla^2 w + \frac{\Delta w}{n - 2} g \geq \frac{4 R^2}{(n - 2) \big( R^2 - |x|^2 \big)^2 } g. \]
It follows that
\[  (n - 2) e^{- 2 w}  \mathcal{W} [ w ] \geq e^{2 C} 4 R^2 g >  e^{2 C} 2 R^2 g . \]
Choosing $C > 0$ sufficiently large such that
\begin{equation} \label{eq4-31}
\begin{aligned}
& \frac{\sigma_k}{\sigma_{k - 1}} \bigg( (n - 2) e^{- 2 w}  \mathcal{W} [ w ] \bigg) - \frac{\alpha_0(x) }{\sigma_{k - 1} \bigg( (n - 2) e^{- 2 w}  \mathcal{W} [ w ]  \bigg)} \\
\geq &  \frac{\sigma_k}{\sigma_{k - 1}} \bigg(   e^{2 C} 2 R^2 g  \bigg) - \frac{\alpha_0(x) }{\sigma_{k - 1} \bigg(   e^{2 C} 2 R^2 g  \bigg)}
> - \alpha(x) \quad \text{  on  } \overline{\Omega}
\end{aligned}
\end{equation}
and
\[ w \leq - \ln \big( R^2 - r^2 \big) - C < \varphi \quad \text{  on  } \partial \Omega. \]

Also, we consider
\[ \eta = 2 \ln \delta - \ln (\rho + \delta^2) + \varphi.  \]
For $\delta > 0$ sufficiently small,
\[ \mathcal{W}[ \eta ] \geq \nabla^2 \eta  +  \frac{\Delta \eta}{n - 2} g  \geq \frac{1}{2(n - 2) (\rho + \delta^2)^2} g \quad \text{  on  } \{ 0 \leq \rho \leq \delta \}.  \]
Consequently,
\[  (n - 2) e^{- 2 \eta}  \mathcal{W}[\eta ] \geq \frac{e^{- 2 \varphi}}{2 \delta^4} g >  \frac{e^{- 2 \varphi}}{4 \delta^4} g \quad \text{  on  } \{ 0 \leq \rho \leq \delta \}.  \]
Choosing $\delta > 0$ further small such that
\begin{equation} \label{eq4-32}
\begin{aligned}
& \frac{\sigma_k}{\sigma_{k - 1}} \bigg( (n - 2) e^{- 2 \eta}  \mathcal{W} [ \eta ] \bigg) - \frac{\alpha_0(x)}{\sigma_{k - 1} \bigg( (n - 2) e^{- 2 \eta}  \mathcal{W} [ \eta ]  \bigg)} \\
\geq &  \frac{\sigma_k}{\sigma_{k - 1}} \bigg(   \frac{e^{- 2 \varphi}}{4 \delta^4} g  \bigg) - \frac{\alpha_0(x)}{\sigma_{k - 1} \bigg(   \frac{e^{- 2 \varphi}}{4 \delta^4} g  \bigg)} > - \alpha(x) \quad \text{  on  } \{ 0 \leq \rho \leq \delta \},
\end{aligned}
\end{equation}
and
\[ \eta < 2 \ln \delta - \ln \delta + \varphi < w  \quad \text{  on  } \rho = \delta. \]

Now we need a lemma from Guan \cite{GuanP02}.
\begin{lemma}
For all $\epsilon > 0$, there is an even function $h(t) \in C^{\infty} (\mathbb{R})$ such that
\begin{enumerate}
  \item $h(t) \geq |t|$ for all $t \in \mathbb{R}$, and $h(t) = |t|$ for all $|t| \geq \epsilon$;
  \item $\big| h'(t) \big| \leq 1$ and $h''(t) \geq 0$ for all $t \in \mathbb{R}$, and $h'(t) \geq 0$ for all $t \geq 0$.
\end{enumerate}
\end{lemma}

Define
\[ \underline{v} = \left\{ \begin{aligned} & \frac{1}{2} ( \eta + w ) + \frac{1}{2} h(\eta - w) \quad \text{ on  } \{ 0 \leq \rho \leq \delta \} \\
& w \quad \text{  on  } \{ \rho \geq \delta \}
\end{aligned} \right.. \]
Direct calculation shows that on $\{ 0 \leq \rho \leq \delta \}$,
\[ \nabla_i \underline{v} = \frac{1}{2} \Big( \nabla_i \eta + \nabla_i w \Big) + \frac{1}{2} h'(\eta - w) \Big( \nabla_i \eta - \nabla_i w \Big),  \]
\[\begin{aligned}
\nabla_{ij} \underline{v} = & \frac{1}{2} \Big( \nabla_{ij} \eta + \nabla_{ij} w \Big)  + \frac{1}{2} h' \Big( \nabla_{ij} \eta - \nabla_{ij} w \Big) \\
& + \frac{1}{2} h'' \Big( \nabla_i \eta - \nabla_i w \Big) \Big( \nabla_j \eta - \nabla_j w \Big).
\end{aligned} \]
It follows that within $\{ 0 \leq \rho \leq \delta \}$,
\[\begin{aligned}
\mathcal{W}[ \underline{v} ] \geq  & \nabla^2 \underline{v} + \frac{\Delta \underline{v}}{n - 2} g  \\
\geq & \frac{1}{2} \Big( \nabla^2 \eta + \nabla^2 w \Big)  + \frac{1}{2} h' \Big( \nabla^2 \eta - \nabla^2 w \Big) \\
& +  \frac{1}{2(n - 2)} \Big( \Delta \eta + \Delta w \Big) g  + \frac{1}{2 (n - 2)} h' \Big( \Delta \eta - \Delta w \Big) g \\
= & \frac{1 + h'}{2} \bigg( \nabla^2 \eta + \frac{\Delta \eta}{n - 2} g \bigg) + \frac{1 - h'}{2} \bigg( \nabla^2 w + \frac{\Delta w}{n - 2} g \bigg) \\
\geq & \frac{1 + h'}{2} \frac{1}{2(n - 2) (\rho + \delta^2)^2} g +  \frac{1 - h'}{2}  \frac{4 R^2}{(n - 2) \big( R^2 - |x|^2 \big)^2 } g.
\end{aligned} \]
Since on $\{ 0 \leq \rho \leq \delta \}$,
\[ \underline{v} \leq \sup \{ \eta, w \} + \frac{\epsilon}{2}, \]
we have
\[\begin{aligned}
& (n - 2) e^{- 2 \underline{v}} \mathcal{W}[ \underline{v} ]  \\
\geq &  (n - 2) e^{- 2 \underline{v}} \frac{1 + h'}{2} \frac{1}{2(n - 2) (\rho + \delta^2)^2} g +   (n - 2) e^{- 2 \underline{v}}   \frac{1 - h'}{2}  \frac{4 R^2}{(n - 2) \big( R^2 - |x|^2 \big)^2 } g \\
\geq &  \left\{ \begin{aligned}
(n - 2) e^{- 2 \eta - \epsilon} \frac{1}{2} \frac{1}{2(n - 2) (\rho + \delta^2)^2} g \quad \text{  on  } \{0 \leq \rho \leq \delta \} \cap \{ \eta \geq w \} \\
(n - 2) e^{- 2 w - \epsilon}   \frac{1}{2}  \frac{4 R^2}{(n - 2) \big( R^2 - |x|^2 \big)^2 } g \quad \text{  on  } \{0 \leq \rho \leq \delta \} \cap \{ \eta \leq w \}
\end{aligned} \right. \\
\geq & \left\{ \begin{aligned}  \frac{1}{4 \delta^4} e^{- 2 \varphi - \epsilon} g \quad \text{  on  } \{0 \leq \rho \leq \delta \} \cap \{ \eta \geq w \} \\
 e^{2 C - \epsilon} 2 R^2 g \quad \text{  on  } \{0 \leq \rho \leq \delta \} \cap \{ \eta \leq w \}
\end{aligned} \right..
\end{aligned} \]

In view of \eqref{eq4-31} and \eqref{eq4-32}, we may choose $\epsilon > 0$ sufficiently small such that
\[  \frac{\sigma_k}{\sigma_{k - 1}} \bigg( (n - 2) e^{- 2 \underline{v}}  \mathcal{W} [ \underline{v} ] \bigg) - \frac{\alpha_0(x) }{\sigma_{k - 1} \bigg( (n - 2) e^{- 2 \underline{v}}  \mathcal{W} [ \underline{v} ]  \bigg)} > - \alpha(x) \quad \text{  on  } \overline{\Omega}.
\]

\vspace{2mm}

\subsection{Preliminary estimates}

\vspace{2mm}

For a $(k - 1)$-admissible function $v \in C^2 (\overline{\Omega})$ with
\[ v \geq \underline{v}  \text{  in  }  \overline{\Omega} \quad \text{  and  }
v = \varphi  \text{  on  }  \partial \Omega, \]
we change $v$ back into $u$ by $v = \frac{2}{n - 2} \ln u$ to see that
\[ \sigma_1 \big( W[u] \big) = 2 (n - 1) \Delta u \geq 0. \]
Let $h$ be the solution to
\[ \left\{ \begin{aligned}
\Delta h = &  0 \quad & \text{  in  } \Omega, \\
h = & e^{\frac{n - 2}{2} \varphi} \quad & \text{  on  } \partial\Omega.
\end{aligned} \right. \]
By the maximum principle, $u \leq h$ in $\Omega$ and hence
\[ v \leq \overline{v} : = \frac{2}{n - 2} \ln h \quad \text{  in  } \Omega. \]
Then we have
\[ \nabla_{\nu} \underline{v} \leq \nabla_{\nu} v \leq \nabla_{\nu} \overline{v} \quad \text{  on  } \partial \Omega,   \]
where $\nu$ is the interior unit normal to $\partial\Omega$.

\vspace{2mm}

\subsection{Existence of solutions}

\vspace{2mm}

Denote
\[\begin{aligned}
G[v]: = & G( \nabla^2 v, \nabla v, v ) = F \big( \mathcal{W} [ v ] \big), \\
G^{ij}[v] : = & G^{ij} ( \nabla^2 v, \nabla v, v ) = \frac{\partial G}{\partial \nabla_{ij} v}, \\
G^{i}[v] : = & G^{i} ( \nabla^2 v, \nabla v, v ) = \frac{\partial G}{\partial \nabla_{i} v}, \\
G_v [v] : = & G_v ( \nabla^2 v, \nabla v, v ) = \frac{\partial G}{\partial v}.
\end{aligned} \]

Let $C_0$ be a positive constant such that
\begin{equation} \label{eq5-1}
G[\underline{v}] =  G( \nabla^2 \underline{v}, \nabla \underline{v}, \underline{v} ) > - C_0 \quad \text{ in } \overline{\Omega}.
\end{equation}
For $t \in [0, 1]$, we consider the following two equations (similar construction of the equations can be found in Su \cite{Su16}).
\begin{equation} \label{eq5-2}
\left\{ \begin{aligned}
G [v]  =  &    ( 1 - t ) G[\underline{v}] - t C_0   \quad  & \text{  in  }  \Omega, \\
v = &  \varphi \quad  & \text{  on  }  \partial\Omega.
\end{aligned} \right.
\end{equation}
\begin{equation} \label{eq5-3}
\left\{ \begin{aligned}
G [v]  =  &  - ( 1 - t ) C_0 -  t \frac{\alpha(x)}{n - 2} e^{2 v}   \quad & \text{  in  }  \Omega, \\
v = &  \varphi  \quad & \text{  on  }  \partial\Omega.
\end{aligned} \right.
\end{equation}

\begin{rem} \label{Remark5-1}
For $x \in \overline{\Omega}$ and a $C^2$ function $v$ which is $(k - 1)$-admissible near $x$, we have
\[ G_v (x) = - \frac{\alpha_0 (x) e^{2 k v} 2 k}{(n - 2)^k} \frac{1}{\sigma_{k - 1} \big( \mathcal{W}[v] \big)} < 0 \]
if $\alpha_0 (x) > 0$.
\end{rem}

\begin{lemma}  \label{Lemma5-1}
For $t \in [0, 1]$,  let $\underline{V}$ and $v$ be any $(k - 1)$-admissible subsolution and solution of \eqref{eq5-2}. Then $v \geq \underline{V}$ in $\Omega$. In particular, \eqref{eq5-2} has at most one $(k - 1)$-admissible solution.
\end{lemma}
\begin{proof}
Suppose that $\underline{V} - v$ achieves a positive maximum at $x_0 \in \Omega$, at which we have
\[\underline{V}(x_0) > v(x_0),\quad \nabla \underline{V}(x_0) = \nabla v(x_0), \quad \nabla^2\underline{V}(x_0) \leq \nabla^2 v(x_0). \]
It follows that $\mathcal{W}[\underline{V}](x_0) \leq \mathcal{W}[v](x_0)$,
and therefore
\[ \begin{aligned}
F \big( \mathcal{W}[\underline{V}] \big) (x_0) = & \frac{\sigma_k}{\sigma_{k - 1}} \big( \mathcal{W} [ \underline{V} ] \big)(x_0) - \frac{\alpha_0(x_0) e^{2 k \underline{V} (x_0)}}{(n - 2)^k} \frac{1}{\sigma_{k - 1} \big( \mathcal{W} [ \underline{V} ]  \big) (x_0) } \\
\leq &  \frac{\sigma_k}{\sigma_{k - 1}} \big( \mathcal{W} [ v ] \big)(x_0) - \frac{\alpha_0(x_0) e^{2 k \underline{V} (x_0)}}{(n - 2)^k} \frac{1}{\sigma_{k - 1} \big( \mathcal{W} [ v ]  \big) (x_0) } \\
< &  \frac{\sigma_k}{\sigma_{k - 1}} \big( \mathcal{W} [ v ] \big)(x_0) - \frac{\alpha_0(x_0) e^{2 k v (x_0)}}{(n - 2)^k} \frac{1}{\sigma_{k - 1} \big( \mathcal{W} [ v ]  \big) (x_0) } \\
= & F \big( \mathcal{W}[v] \big) (x_0).
\end{aligned} \]
But
\[ F \big( \mathcal{W}[\underline{V}] \big) (x_0) \geq ( 1 - t ) G[\underline{v}] (x_0) - t C_0 =  F \big( \mathcal{W}[v] \big) (x_0),   \]
which is a contradiction.
\end{proof}

\begin{thm} \label{Theorem5-1}
For $t \in [0, 1]$, \eqref{eq5-2} has a unique $(k - 1)$-admissible solution $v \geq \underline{v}$.
\end{thm}
\begin{proof}
By Lemma \ref{Lemma5-1}, we immediately obtain uniqueness. For existence, we use standard continuity method. By assumption \eqref{eq5-1},  $\underline{v}$ is a subsolution of \eqref{eq5-2}. We note that the $C^2$ estimate for $(k - 1)$-admissible solution $v \geq \underline{v}$ of \eqref{eq5-2} implies uniform ellipticity of this equation and hence gives $C^{2, \alpha}$ estimate by Evans-Krylov theory \cite{Evans, Krylov}
\begin{equation} \label{eq5-4}
\Vert v \Vert_{C^{2, \alpha} ( \overline{ \Omega } )}  \leq C,
\end{equation}
where $C$ is independent of $t$. Denote
\[ C_0^{2, \alpha} ( \overline{ \Omega } ) = \{ w \in C^{2, \alpha}(  \overline{ \Omega }  ) \,| \, w = 0  \text{  on  }  \partial\Omega \}, \]
and consider the open subset of $C_0^{2, \alpha} (\overline{ \Omega })$
\[ \mathcal{U} = \{ w \in C_0^{2, \alpha} ( \overline{ \Omega } ) \,| \, \underline{v} + w \text{ is } (k - 1)\text{-admissible in } \overline{\Omega}  \}. \]
Define $\mathcal{L}: \mathcal{U} \times [ 0, 1 ] \rightarrow C^{\alpha}( \overline{ \Omega } )$,
\[ \mathcal{L} ( w, t ) = G [ \underline{v} + w ] - ( 1 - t ) G[\underline{v}] + t C_0,  \]
and set
\[ \mathcal{S} = \{ t \in [0, 1] \,|\, \mathcal{L}(w, t) = 0 \text{ has a solution } w \text{ in } \mathcal{U} \}. \]
$\mathcal{S} \neq \emptyset$ since $\mathcal{L}(0, 0) = 0$.

$\mathcal{S}$ is open in $[0, 1]$. In fact, for any $t_0 \in \mathcal{S}$, there exists $w_0 \in \mathcal{U}$ such that $\mathcal{L} (w_0, t_0) = 0$. Note that the Fr\'echet derivative of $\mathcal{L}$ with respect to $w$ at $(w_0, t_0)$ is a linear elliptic operator from $C^{2, \alpha}_0 ( \overline{\Omega} )$ to $C^{\alpha}( \overline{\Omega})$,
\[ \mathcal{L}_w \big|_{(w_0, t_0)} ( h )  = G^{ij}[\underline{v} + w_0] \nabla_{ij} h  +  G^i [ \underline{v} + w_0 ] \nabla_i h  + G_v [ \underline{v} + w_0] h. \]
Remark \ref{Remark5-1} implies $\mathcal{L}_w \big|_{(w_0, t_0)}$ is invertible. Thus a neighborhood of $t_0$ is also contained in $\mathcal{S}$ by implicit function theorem.

$\mathcal{S}$ is closed in $[0, 1]$. In fact, let $t_i$ be a sequence in $\mathcal{S}$ converging to $t_0 \in [0, 1]$ and $w_i \in \mathcal{U}$ be the unique solution to $\mathcal{L} (w_i, t_i) = 0$. Lemma \ref{Lemma5-1} implies $w_i \geq 0$. By \eqref{eq5-4}, $v_i := \underline{v} + w_i$ is a bounded sequence in $C^{2, \alpha}(\overline{\Omega})$, which possesses a subsequence converging to a $(k - 1)$-admissible solution $v_0$ of \eqref{eq5-2}. Since $w_0 = v_0 - \underline{v} \in \mathcal{U}$ and $\mathcal{L}(w_0, t_0) = 0$, we know that $t_0 \in \mathcal{S}$.
\end{proof}

Next we may assume that $\underline{v}$ is not a solution of \eqref{eq4-30}.

\begin{lemma} \label{Lemma5-2}
If $v \geq \underline{v}$ is a $(k - 1)$-admissible solution of \eqref{eq5-3}, then
$v > \underline{v}$ in $\Omega$ and $\nabla_{\nu}(v - \underline{v}) > 0$ on $\partial\Omega$.
\end{lemma}

\begin{proof}
We note that
\[ G[\underline{v}] - G[v] \geq - t \frac{\alpha(x)}{n - 2} (e^{2 \underline{v}} - e^{2 v}), \]
and
\[\begin{aligned}
& G [\underline{v}] - G [v] = F \big( \mathcal{W}[\underline{v}] \big) -  F \big( \mathcal{W}[v] \big) \\
= & \int_0^1 \frac{d}{d s} \Bigg(  \frac{\sigma_k}{\sigma_{k - 1}}  \Big( (1 - s) \mathcal{W} [ v ] + s \mathcal{W}[\underline{v}] \Big) - \frac{\alpha_0(x) e^{2 k ( (1 - s ) v + s \underline{v} )}}{(n - 2)^k \sigma_{k - 1} \Big( (1 - s) \mathcal{W} [ v ] + s \mathcal{W}[\underline{v}] \Big) } \Bigg) d s \\
= & \underbrace{\int_0^1 \Bigg( \Big( \frac{\sigma_k}{\sigma_{k - 1}} \Big)^{ij} +  \frac{\alpha_0(x) e^{2 k ( (1 - s ) v + s \underline{v} )} \sigma_{k - 1}^{ij}}{(n - 2)^k \sigma_{k - 1}^2 } \Bigg) \Big( (1 - s) \mathcal{W}[v] + s \mathcal{W}[\underline{v}] \Big)  d s }_{\text{denoted by } a_{ij}} \\
& \cdot \Big( \mathcal{W}_{ij}[\underline{v}] - \mathcal{W}_{ij}[v] \Big) - \underbrace{\int_0^1 \frac{\alpha_0(x) e^{2 k ( (1 - s ) v + s \underline{v} )} 2 k}{(n - 2)^k \sigma_{k - 1} \Big( (1 - s) \mathcal{W} [ v ] + s \mathcal{W}[\underline{v}] \Big) } d s }_{\text{denoted by } c} \cdot (\underline{v} - v) \\
= & a_{ij} \bigg( \nabla_{ij} (\underline{v} - v) + \frac{\Delta(\underline{v} - v)}{n - 2} \delta_{ij} \bigg) - a_{ij} \nabla_i \underline{v} \nabla_j (\underline{v} - v) - a_{ij} \nabla_j v \nabla_i (\underline{v} - v) \\
& + \nabla (\underline{v} + v) \cdot \nabla (\underline{v} - v) \sum a_{ii}  - c (\underline{v} - v).
\end{aligned} \]
Applying the Maximum Principle and Lemma H (see p. 212 of \cite{GNN}) we proved the lemma.
\end{proof}

\begin{thm} \label{Theorem5-2}
For any $t \in [0, 1]$, there is a $(k - 1)$-admissible solution $v \geq \underline{v}$ to Dirichlet problem \eqref{eq5-3}.
\end{thm}

\begin{proof}
We obtain by classical Schauder theory the $C^{4, \alpha}$ estimate
\begin{equation} \label{eq5-5}
\Vert v \Vert_{C^{4,\alpha}(\overline{\Omega})} < C_4.
\end{equation}
Also, we have
\begin{equation} \label{eq5-6}
\text{dist} \big( \lambda(\mathcal{W}[v]), \partial\Gamma_{k - 1} \big) > c_2 > 0 \quad \text{  in  }  \overline{\Omega},
\end{equation}
where $C_4$ and $c_2$ are independent of $t$. Denote
\[ C_0^{4, \alpha} (\overline{\Omega}) = \big\{ w \in C^{4, \alpha}( \overline{\Omega}) \,|\, w = 0  \text{  on }  \partial\Omega \big\} \]
and the open bounded subset of $C_0^{4, \alpha} (\overline{\Omega})$
\[ \mathcal{O} = \Bigg\{ w \in C_0^{4, \alpha} (\overline{\Omega}) \left\vert \begin{footnotesize}\begin{aligned} & w > 0 \text{  in  } \Omega,  \nabla_{\nu} w > 0 \text{  on  } \partial\Omega,  \Vert w {\Vert}_{C^{4,\alpha}(\overline{\Omega})} < C_4 + \Vert\underline{v} \Vert_{C^{4,\alpha}(\overline{\Omega})}, \\
& \underline{v} + w  \text{ is } (k - 1) \text{-admissible in }  \overline{\Omega}, \text{dist} \Big(\lambda \big( \mathcal{W} [\underline{v} + w] \big), \partial \Gamma_{k - 1} \Big) > c_2   \text{ in }  \overline{\Omega}  \end{aligned}\end{footnotesize} \right.\Bigg\}. \]
Define a map
$\mathcal{M}_t (w):  \mathcal{O} \times [ 0, 1 ] \rightarrow C^{2, \alpha}(\overline{\Omega})$,
\[ \mathcal{M}_t (w) = G [ \underline{v} + w ]  + ( 1 - t ) C_0 + t \frac{\alpha(x)}{n - 2} e^{2 (\underline{v} + w)}. \]
By Theorem \ref{Theorem5-1} and Lemma \ref{Lemma5-1}, there is a unique $(k - 1)$-admissible solution $v^0$ of \eqref{eq5-2} at $t = 1$, which is also the unique $(k - 1)$-admissible solution of \eqref{eq5-3} at $t = 0$. By Lemma \ref{Lemma5-1}, $w^0 = v^0 - \underline{v} \geq 0$ in $\Omega$. Consequently, $w^0 > 0$ in $\Omega$ and $\nabla_{\nu} w^0 > 0$ on $\partial\Omega$ by Lemma \ref{Lemma5-2}. We also note that $\underline{v} + w^0$ satisfies \eqref{eq5-5} and \eqref{eq5-6}. Thus, $w^0 \in \mathcal{O}$. By Lemma \ref{Lemma5-2}, \eqref{eq5-5} and \eqref{eq5-6}, $\mathcal{M}_t(w) = 0$ has no solution on $\partial\mathcal{O}$ for any $t \in [0, 1]$. Since $\mathcal{M}_t$ is uniformly elliptic on $\mathcal{O}$ independent of $t$, the degree of $\mathcal{M}_t$ on $\mathcal{O}$ at $0$ can be defined independent of $t$. This degree is nonzero at $t = 0$. In fact,
since $\mathcal{M}_0 ( w ) = 0$ has a unique solution $w^0 \in \mathcal{O}$. The Fr\'echet derivative of $\mathcal{M}_0$ with respect to $w$ at $w^0$ is a linear elliptic operator from $C^{4, \alpha}_0 (\overline{\Omega})$ to $C^{2, \alpha}(\overline{\Omega})$,
\begin{equation*}
\mathcal{M}_{0, w} |_{w^0} ( h )  =  G^{ij}[ v^0 ] \nabla_{ij} h  + G^i [ v^0 ] \nabla_i h + G_v [ v^0 ] h.
\end{equation*}
By Remark \ref{Remark5-1}, $G_v [ v^0 ] < 0$ in $\overline{\Omega}$. Thus, $\mathcal{M}_{0, w} |_{w^0}$ is invertible. The degree theory in \cite{Li89} implies that the degree at $t = 0$ is nonzero, which implies that \eqref{eq5-3} has at least one $(k - 1)$-admissible solution $v \geq \underline{v}$ for any $t \in [0, 1]$.
\end{proof}

\vspace{4mm}

\section{Fully nonlinear Loewner-Nirenberg problem for general equations}

\vspace{4mm}

In this section, we discuss fully nonlinear Loewner-Nirenberg problem related to general equation \eqref{eq4-1}.

\vspace{2mm}

\subsection{The case of smooth compact boundary consisting of closed hypersurfaces}~

\vspace{2mm}

We first give the definition of supersolutions and subsolutions to equation \eqref{eq4-1}.

\begin{defn}
A function $0 < \underline{u} \in C^2(\Omega)$ is a subsolution of \eqref{eq4-1} in $\Omega$ if
\[\begin{aligned}
& \lambda\big( W [\underline{u}] \big) \in \Gamma_{k - 1}  \text{    and    }  \\
\sigma_k \bigg( \frac{2}{n - 2} \underline{u}^{- \frac{n + 2}{n - 2}}  W[ \underline{u} ] \bigg) & + \alpha(x) \sigma_{k - 1} \bigg( \frac{2}{n - 2} \underline{u}^{- \frac{n + 2}{n - 2}} W[ \underline{u} ] \bigg) \geq \alpha_0(x) \quad \text{  in  }  \Omega.
\end{aligned} \]
A function $0 < \overline{u} \in C^2(\Omega)$ is a supersolution of \eqref{eq4-1} in $\Omega$ if
\[\begin{aligned}
& \text{either  } \lambda\big( W [\overline{u}] \big) \notin \Gamma_{k - 1}  \text{  or  } \\
\sigma_k \bigg( \frac{2}{n - 2} \overline{u}^{- \frac{n + 2}{n - 2}}  W[ \overline{u} ] \bigg) & + \alpha(x) \sigma_{k - 1} \bigg( \frac{2}{n - 2} \overline{u}^{- \frac{n + 2}{n - 2}} W[ \overline{u} ] \bigg) \leq \alpha_0(x) \quad \text{  in  }  \Omega.
\end{aligned} \]
\end{defn}

\begin{prop} \label{Prop6-1}
If $\underline{u}$ is a subsolution of \eqref{eq4-1} satisfying
\begin{equation} \label{eq6-4}
\lambda\big( W [\underline{u}] \big) \in \overline{\Gamma}_k \quad \text{  in  } \Omega,
\end{equation}
so does $c \underline{u}$ for any constant $0 < c < 1$. If $\overline{u}$ is a supersolution of \eqref{eq4-1} satisfying
\[  \lambda\big( W [\overline{u}] \big) \notin \Gamma_{k - 1} \quad \text{everywhere in  } \Omega, \]
then $c \overline{u}$ is a supersolution of \eqref{eq4-1} for any $c > 0$.
\end{prop}
\begin{proof}
When $\lambda\big( W[ \underline{u}] \big) \in \overline{\Gamma}_k$,
\[
\begin{aligned}
& \frac{\sigma_k}{\sigma_{k - 1}} \bigg( \frac{2}{n - 2} c^{- \frac{4}{n - 2}} \underline{u}^{- \frac{n + 2}{n - 2}}  W[ \underline{u} ] \bigg) - \frac{ \alpha_0(x)}{\sigma_{k - 1} \bigg( \frac{2}{n - 2} c^{- \frac{4}{n - 2}} \underline{u}^{- \frac{n + 2}{n - 2}}  W[ \underline{u} ] \bigg)} \\
\geq & \frac{\sigma_k}{\sigma_{k - 1}} \bigg( \frac{2}{n - 2} \underline{u}^{- \frac{n + 2}{n - 2}}  W[ \underline{u} ] \bigg) - \frac{ \alpha_0(x)}{\sigma_{k - 1} \bigg( \frac{2}{n - 2} \underline{u}^{- \frac{n + 2}{n - 2}}  W[ \underline{u} ] \bigg)} \geq - \alpha(x).
\end{aligned}
\]
Hence we proved the first statement. The second statement is straightforward by definition.
\end{proof}

Then we observe the following maximum principle of \eqref{eq4-1}.
\begin{thm} \label{Theorem6-1}
Let $M$ be a smooth compact manifold with boundary. Suppose that $\underline{u}$ is a subsolution of
\begin{equation*}
\begin{aligned}
\sigma_k \bigg( \frac{2}{n - 2} u^{- \frac{n + 2}{n - 2}}  W[ u ] \bigg) + \alpha (x) \sigma_{k - 1} \bigg( \frac{2}{n - 2} u^{- \frac{n + 2}{n - 2}} W[ u ] \bigg) = \alpha_0(x)
\end{aligned}
\end{equation*}
satisfying
\[
\lambda\big( W [\underline{u}] \big) \in \Gamma_k \quad \text{  in  } \Omega,
\]
and $\overline{u}$ is a supersolution of
\begin{equation*}
\begin{aligned}
\sigma_k \bigg( \frac{2}{n - 2} u^{- \frac{n + 2}{n - 2}}  W[ u ] \bigg) + \beta(x) \sigma_{k - 1} \bigg( \frac{2}{n - 2} u^{- \frac{n + 2}{n - 2}} W[ u ] \bigg) = \beta_0(x),
\end{aligned}
\end{equation*}
with $\alpha(x) \leq \beta(x)$ and $\alpha_0(x) \geq \beta_0(x)$ on $M$. If $\underline{u} \leq \overline{u}$ on $\partial M$, then $\underline{u} \leq \overline{u}$ on $M$.
\end{thm}

\begin{proof}
Suppose that $\underline{u} > \overline{u}$ somewhere in the interior of $M$. Let $C$ be the maximum of $\frac{\underline{u}}{\overline{u}}$ on $M$, which is attained at $x_0$ in the interior of $M$. Since $C > 1$, by Proposition \ref{Prop6-1}, we know that $\underline{u}_1 := \frac{\underline{u}}{C}$ satisfies
\begin{equation*}
 \frac{\sigma_k}{\sigma_{k - 1}} \bigg( \frac{2}{n - 2} \underline{u}_1^{- \frac{n + 2}{n - 2}}  W[ \underline{u}_1 ] \bigg) - \frac{ \alpha_0(x)}{\sigma_{k - 1} \bigg( \frac{2}{n - 2} \underline{u}_1^{- \frac{n + 2}{n - 2}}  W[ \underline{u}_1 ] \bigg)} > - \alpha(x).
\end{equation*}
On the other hand, since $\underline{u}_1(x_0) = \overline{u}(x_0)$ while $\underline{u}_1 \leq \overline{u}$ near $x_0$, thus at $x_0$
\[\nabla \underline{u}_1 (x_0) = \nabla \overline{u} (x_0), \quad \nabla^2 \underline{u}_1 (x_0) \leq \nabla^2 \overline{u} (x_0). \]
It follows that $W[\underline{u}_1](x_0) \leq W[\overline{u}](x_0)$, which further implies
\[ \begin{aligned}
& - \alpha (x_0) < \frac{\sigma_k}{\sigma_{k - 1}} \bigg( \frac{2}{n - 2} \underline{u}_1^{- \frac{n + 2}{n - 2}}  W[ \underline{u}_1 ] \bigg)(x_0) - \frac{ \alpha_0(x_0)}{\sigma_{k - 1} \bigg( \frac{2}{n - 2} \underline{u}_1^{- \frac{n + 2}{n - 2}}  W[ \underline{u}_1 ] \bigg)(x_0)} \\
\leq & \frac{\sigma_k}{\sigma_{k - 1}} \bigg( \frac{2}{n - 2} \overline{u}^{- \frac{n + 2}{n - 2}}  W[ \overline{u} ] \bigg)(x_0) - \frac{ \beta_0(x_0)}{\sigma_{k - 1} \bigg( \frac{2}{n - 2} \overline{u}^{- \frac{n + 2}{n - 2}}  W[ \overline{u} ] \bigg)(x_0)} \leq - \beta (x_0),
\end{aligned} \]
which is a contradiction.
\end{proof}

Similar to Proposition \ref{subsolution}, we observe the following fact.
\begin{prop} \label{Proposition6-1}
For any real number $\alpha$, any positive constant $\alpha_0$, any fixed $s > 0$, there exists a positive constant $c$ such that
\begin{equation*}
 u (x) = c \Big( s^2 - {|x - x_0|}^2 \Big)^{1 - \frac{n}{2}} \quad \text{ in  }  B_s (x_0),
\end{equation*}
and
\begin{equation*}
v (x) = c \Big( {|x - x_0|}^2 - s^2 \Big)^{1 - \frac{n}{2}} \quad \text{  in  } \mathbb{R}^n\setminus \overline{B_s (x_0)},
\end{equation*}
are admissible solutions of
\[
\begin{aligned}
\sigma_k \bigg( \frac{2}{n - 2} u^{- \frac{n + 2}{n - 2}}  W[ u ] \bigg) + \alpha \sigma_{k - 1} \bigg( \frac{2}{n - 2} u^{- \frac{n + 2}{n - 2}} W[ u ] \bigg) = \alpha_0,
\end{aligned}
\]
which approach to $\infty$ on $\partial B_s (x_0)$.
\end{prop}

\begin{proof}
Consider
\[  u(x) =  c \Big( s^2 - {|x - x_0|}^2 \Big)^{1 - \frac{n}{2}}. \]
We have
\[ W_{ij} [u] = 2 (n - 1) (n - 2) c \Big( s^2 - |x - x_0|^2 \Big)^{- \frac{n}{2} - 1} s^2  \delta_{ij}, \]
and therefore
\[ \frac{2}{n - 2} u^{- \frac{n + 2}{n - 2}} W_{ij} [u] = 4 (n - 1) c^{- \frac{4}{n - 2}} s^2 \delta_{ij}. \]
Hence
\[\begin{aligned}
& \sigma_k \bigg( \frac{2}{n - 2} u^{- \frac{n + 2}{n - 2}}  W[ u ] \bigg) + \alpha \sigma_{k - 1} \bigg( \frac{2}{n - 2} u^{- \frac{n + 2}{n - 2}} W[ u ] \bigg)  \\
= & \bigg(  4 (n - 1) c^{- \frac{4}{n - 2}} s^2 \bigg)^k \sigma_k \big( I \big) + \alpha \bigg(  4 (n - 1) c^{- \frac{4}{n - 2}} s^2 \bigg)^{k - 1} \sigma_{k - 1} \big( I \big).
\end{aligned} \]
Since $k \geq 2$, by intermediate value theorem, there exists a positive constant $c$ such that the above quantity equals $\alpha_0$.
\end{proof}

\begin{rem} \label{Remark6-1}
The relation between $c$ and $s$ in $u(x)$ and $v(x)$ in Proposition \ref{Proposition6-1} can be obtained once $\alpha$ and $\alpha_0$ are fixed. In fact, let $\tilde{\mu}$ be a positive root of
\[  \mu^k \sigma_k \big( I \big) + \alpha \mu^{k - 1} \sigma_{k - 1} \big( I \big) = \alpha_0. \]
Then choose
\[  4 (n - 1) c^{- \frac{4}{n - 2}} s^2 = \tilde{\mu}  \]
to obtain
\[ c = \bigg(\frac{4 (n - 1)}{\tilde{\mu}}\bigg)^{\frac{n - 2}{4}} s^{\frac{n - 2}{2}}. \]
\end{rem}

\text{\bf Proof of Theorem \ref{Theorem6-2} }
If $\Omega$ is a bounded domain, for any positive integer $m$, we consider the Dirichlet problem
\begin{equation} \label{eq6-1}
\left\{\begin{aligned}
F \big( \mathcal{W} [ v ] \big) = & - \frac{\alpha(x)}{n - 2} e^{2 v} \quad &\text{  in  } \Omega, \\
v = & m \quad &\text{  on  } \partial\Omega.
\end{aligned} \right.
\end{equation}
By Theorem \ref{Theorem5-2}, there exists a smooth $(k - 1)$-admissible solution $v_m$ to \eqref{eq6-1}.
Let $\mu_0$ be the smallest positive root of the equation
\[ \mu^k + \overline{\alpha} \mu^{k - 1} = \underline{\alpha}_0,  \]
where
\[ \overline{\alpha} = \max \Big\{ \sup\limits_{\Omega} \alpha(x), 0.01 \Big \} > 0, \quad \underline{\alpha}_0 = \inf\limits_{\Omega} \alpha_0(x) > 0. \]
By Loewner and Nirenberg \cite{Loewner}, there exists a positive solution $\overline{u}$ to
\begin{equation} \label{eq6-2}
\left\{ \begin{aligned}
\sigma_1 \bigg( \frac{2}{n - 2} u^{- \frac{n + 2}{n - 2}} W[u] \bigg) = & \mu_0 \quad \text{  in  } \Omega, \\
u = & \infty \quad \text{  on  } \partial \Omega.
\end{aligned} \right.
\end{equation}
Since $u_m = e^{\frac{n - 2}{2} v_m}$ satisfies
\[\begin{aligned}
\underline{\alpha}_0 \leq \alpha_0 (x) = & \sigma_k \bigg( \frac{2}{n - 2} u_m^{- \frac{n + 2}{n - 2}}  W[ u_m ] \bigg) + \alpha(x) \sigma_{k - 1} \bigg( \frac{2}{n - 2} u_m^{- \frac{n + 2}{n - 2}}  W[ u_m ] \bigg) \\
\leq & \sigma_1^k  \bigg( \frac{2}{n - 2} u_m^{- \frac{n + 2}{n - 2}}  W[ u_m ] \bigg) + \overline{\alpha} \sigma_1^{k - 1} \bigg( \frac{2}{n - 2} u_m^{- \frac{n + 2}{n - 2}}  W[ u_m ] \bigg),
\end{aligned} \]
we have
\[ \sigma_1 \bigg( \frac{2}{n - 2} u_m^{- \frac{n + 2}{n - 2}}  W[ u_m ] \bigg) \geq \mu_0. \]
By the maximum principle Theorem \ref{MP}, we have
\[  u_m  \leq \overline{u}  \quad \text{  in  } \Omega \quad \text{  for any  } m. \]

Denote
\[ \underline{\alpha} = \inf\limits_{\Omega} \alpha(x), \quad \overline{\alpha}_0 = \sup\limits_{\Omega} \alpha_0(x). \]
Let $t > 0$ be sufficiently large such that $\overline{\Omega} \subset B_t(0)$.
Consider the equation
\begin{equation} \label{eq6-3}
\sigma_k \bigg(  \frac{2}{n - 2} u^{- \frac{n + 2}{n - 2}}  W[ u ]  \bigg) + \underline{\alpha} \sigma_{k - 1} \bigg(  \frac{2}{n - 2} u^{- \frac{n + 2}{n - 2}}  W[ u ]  \bigg) = \overline{\alpha}_0.
\end{equation}
For any fixed $s > t > 0$, by Proposition \ref{Proposition6-1} and Remark \ref{Remark6-1}, there exists a smooth solution
\[ \underline{u} =  \bigg(\frac{4 (n - 1)}{\tilde{\mu}}\bigg)^{\frac{n - 2}{4}} s^{\frac{n - 2}{2}}  \Big( s^2 - |x|^2 \Big)^{1 - \frac{n}{2}}   \]
to equation \eqref{eq6-3} in $B_s(0)$, which is a subsolution of \eqref{eq4-1}. We may choose $s$ sufficiently large so that
\[\begin{aligned}
\underline{u}  \leq \bigg(\frac{4 (n - 1)}{\tilde{\mu}}\bigg)^{\frac{n - 2}{4}} s^{\frac{n - 2}{2}} \Big( s^2 - t^2 \Big)^{1 - \frac{n}{2}} \leq e^{\frac{n - 2}{2}}
\leq e^{\frac{(n - 2) m}{2}} = u_m  \quad \text{  on  } \partial\Omega.
\end{aligned} \]
By Theorem \ref{Theorem6-1} we know that
\[ u_m \geq \underline{u} \quad \text{  in  } \Omega \quad \text{  for any  } m.  \]
By interior $C^2$ estimate established in last section and Evans-Krylov theory \cite{Evans, Krylov}, we obtain a smooth  positive $(k - 1)$-admissible solution $u$ to equation \eqref{eq4-1} which tends to $\infty$ on $\partial \Omega$.

When $\Omega$ is unbounded, we may assume without loss of generality that $0 \notin \overline{\Omega}$. Let $B_s(0)$ be a fixed ball such that $\overline{\Omega} \subset \mathbb{R}^n\setminus \overline{B_s(0)}$.
By Proposition \ref{Proposition6-1}, there exists a smooth solution
\[ \underline{u} = c_0 \Big( |x|^2 - s^2 \Big)^{1 - \frac{n}{2}} \quad \text{ for some positive constant } c_0  \]
to equation \eqref{eq6-3} in $\mathbb{R}^n \setminus \overline{B_s(0)}$, which is a subsolution of \eqref{eq4-1}.

For any $R > \max\limits_{\partial\Omega} \underline{u}$ large enough such that $\partial\Omega \subset B_R(0)$, by Theorem \ref{Theorem5-2}, there exists a smooth positive $(k - 1)$-admissible solution $u_R$ to the Dirichlet problem
\begin{equation*}
\left\{
\begin{aligned}
& \sigma_k \bigg( \frac{2}{n - 2} u^{- \frac{n + 2}{n - 2}}  W[ u ] \bigg) + \alpha \sigma_{k - 1} \bigg( \frac{2}{n - 2} u^{- \frac{n + 2}{n - 2}} W[ u ] \bigg) = \alpha_0  \quad \text{  in  }   B_{R}(0) \cap \Omega, \\
& u =  R \quad \text{  on  }  \partial \Omega, \\
& u =  \underline{u} \quad \text{  on  }  \partial B_{R}(0).
\end{aligned}
\right.
\end{equation*}
By Theorem \ref{Theorem6-1}, we know that
\[  u_R \geq \underline{u} \quad \text{  in  }  B_R(0) \cap \Omega.   \]

By Lemma \ref{Loewner}, we are able to find a positive solution $\overline{u} \in C^{\infty}(\Omega)$ of
\eqref{eq6-2}
which decays to $0$ at $\infty$ with the decay rate
\[  |x|^{n - 2} \overline{u} (x) \rightarrow  c \quad \text{  as  } |x| \rightarrow \infty \]
for some positive constant $c$.
Comparing the decay rate as $|x| \rightarrow \infty$, we have on $\partial B_R(0)$ with $R$ sufficiently large,
\[ u_R (x) = \underline{u}(x) \leq \frac{2 c_0}{c} \overline{u}(x). \]
Meanwhile, we have
\[ u_R (x) = R < \infty = \frac{2 c_0}{c} \overline{u}(x) \quad \text{  on  } \partial \Omega. \]
By Proposition \ref{prop} and Theorem \ref{MP},  we obtain
\[ u_R(x) \leq \max\Big\{\frac{2 c_0}{c}, 1 \Big\} \overline{u} (x) \quad \text{  in  } B_R(0) \cap \Omega. \]
Applying the interior estimates in section 4, Evans-Krylov interior estimates \cite{Evans, Krylov} and a standard diagonal process, we obtain a smooth positive $(k - 1)$-admissible solution $u$ of \eqref{eq4-1} in $\Omega$ which tends to $\infty$ on $\partial \Omega$.
\hfill \qedsymbol

\vspace{2mm}

\subsection{Maximal solution on general domain in $\mathbb{R}^n$ when $\alpha \leq 0$}~

\vspace{2mm}

We first observe the fact that when $\alpha(x) \leq 0$, any $(k - 1)$-admissible solution to equation \eqref{eq4-1} must be $k$-admissible. Then we are able to apply the maximum principle Theorem \ref{Theorem6-1}. As in section 3, we can define the maximal solution of \eqref{eq4-1}.

Assume throughout this subsection that $\Omega\subset\mathbb{R}^n$ is a domain with smooth compact boundary $\partial \Omega$. Moreover, we assume that
\[\alpha(x) \leq 0  \text{  in  } \Omega, \quad   \underline{\alpha}_0 = \inf\limits_{\Omega} \alpha_0(x) > 0, \quad   \underline{\alpha} = \inf\limits_{\Omega} \alpha(x) > - \infty, \quad \overline{\alpha}_0 = \sup\limits_{\Omega} \alpha_0(x) < \infty.   \]

\begin{defn} \label{Maximal solution}
A smooth positive $k$-admissible solution $u_{\Omega}$ of \eqref{eq4-1} is said to be maximal in $\Omega$, if it is greater than or equal to any smooth positive $k$-admissible solution of \eqref{eq4-1} in $\Omega$.
\end{defn}

Let $\Omega_{(1)} \Subset \Omega_{(2)} \Subset \ldots$ be an increasing sequence of bounded subdomains of $\Omega$ with smooth compact boundaries $\partial\Omega_{(j)}$ which are closed hypersurfaces such that $\Omega = \cup \Omega_{(j)}$.
By Theorem \ref{Theorem6-2}, we can find a smooth positive $k$-admissible solution $u_{(j)}$ of \eqref{eq4-1} in $\Omega_{(j)}$ which tends to $\infty$ on $\partial \Omega_{(j)}$.
By Theorem \ref{Theorem6-1}, we see that $\{u_{(j)}\}$ is a monotone decreasing sequence of positive functions. It follows that $u_{(j)}$ converges to a nonnegative function $u_{\Omega}$ in $\Omega$.

\begin{lemma} \label{Lemma6-1}
Either $u_{\Omega} > 0$ in $\Omega$ or $u_{\Omega} \equiv 0$ in $\Omega$.
\end{lemma}

\begin{proof}
Suppose that $u_{\Omega} \not\equiv 0$. Then we must have $u_{\Omega} > c$ on some $B(x_0, r_0) \subset \Omega$ for some constant $c > 0$. By Proposition \ref{Proposition6-1} and Remark \ref{Remark6-1},
we may choose $0 < r_1 < r_0$ such that
\[ v =  \bigg(\frac{4 (n - 1)}{\tilde{\mu}}\bigg)^{\frac{n - 2}{4}} r_1^{\frac{n - 2}{2}} \Big( {|x - x_0|}^2 - r_1^2 \Big)^{1 - \frac{n}{2}} \quad \text{  in  } \mathbb{R}^n\setminus \overline{B_{r_1} (x_0)} \]
is a solution of
\[ \sigma_k \bigg(  \frac{2}{n - 2} u^{- \frac{n + 2}{n - 2}}  W[ u ]  \bigg) + \underline{\alpha} \sigma_{k - 1} \bigg(  \frac{2}{n - 2} u^{- \frac{n + 2}{n - 2}}  W[ u ]  \bigg) = \overline{\alpha}_0 \]
satisfying
\[   \bigg(\frac{4 (n - 1)}{\tilde{\mu}}\bigg)^{\frac{n - 2}{4}} r_1^{\frac{n - 2}{2}}  \Big( r_0^2 - r_1^2 \Big)^{1 - \frac{n}{2}} = c.  \]
By Theorem \ref{Theorem6-1}, we know that $u_{(j)} \geq v$ in $\Omega_{(j)}\setminus B(x_0, r_0)$ for any $j$. Hence $u_{\Omega} \geq v > 0$ in $\Omega\setminus B(x_0, r_0)$. We thus have $u_{\Omega} > 0$ in $\Omega$.
\end{proof}

\vspace{1mm}

\begin{rem}~
\begin{enumerate}
\item In view of Lemma \ref{Lemma6-1}, the interior regularity in section 4, and Evans-Krylov theory
\cite{Evans, Krylov}, we know that $u_{\Omega}$ is smooth.

\item If there exists a positive $k$-admissible subsolution of \eqref{eq4-1} in $\Omega$ (for example, when $\mathbb{R}^n \setminus \overline{\Omega} \neq \emptyset$), then $u_{\Omega} > 0$ in $\Omega$.

\item When $u_{\Omega} > 0$ in $\Omega$, $u_{\Omega}$ is the maximal solution of equation \eqref{eq4-1} in $\Omega$. In fact, if $w$ is any smooth positive $k$-admissible solution of \eqref{eq4-1} in $\Omega$, then $u_{(j)} \geq w$ in $\Omega_{(j)}$ by Theorem \ref{Theorem6-1}. Hence $u_{\Omega} \geq  w$ in $\Omega$.
\end{enumerate}
\end{rem}

\begin{defn} \label{def6-1}
 We call a compact subset $\Gamma \subset \partial\Omega$ regular, if
\begin{equation*}
u_{\Omega} (x) \rightarrow \infty  \quad \text{  as  }  x \rightarrow \Gamma.
\end{equation*}
\end{defn}

Now we consider a portion $\Gamma\subset\partial\Omega$ which is a smooth compact non-self-intersecting surface of codimension $m$. Meanwhile, we assume that $\partial\Omega\setminus\Gamma$ is smooth compact.

\begin{thm} \label{Theorem6-3}
Let $\Omega$ be a domain in $\mathbb{R}^n$ and $\Gamma$ be a compact subset of $\partial\Omega$ such that $\partial\Omega\setminus\Gamma$ is also compact. Suppose that $u_{\Omega} \not\equiv 0$. Then
$\Gamma$ is regular if there exists an open neighborhood $U$ of $\Gamma$ and a $C^2$ positive $k$-admissible subsolution $\phi(x)$ of \eqref{eq4-1} defined in $\Omega \cap U$ which tends to $\infty$ as $x \rightarrow \Gamma$.
\end{thm}

\begin{proof}
Without loss of generality we may assume that $\overline{U}$ is compact,
\[ \overline{U} \cap (\partial \Omega \setminus \Gamma) = \emptyset  \text{  and  }  0 < \phi \in C^2(\overline{U} \cap \Omega). \]
For $j$ sufficiently large, we have
\[ \partial (U \cap \Omega_{(j)}) = \partial U  \cup ( U \cap \partial \Omega_{(j)} ). \]
Since $u_{\Omega}$ is positive in $\Omega$ and $\partial U \subset \Omega$ is compact, we have
$\inf\limits_{\partial U} u_{\Omega} := m > 0$.
Denote $M := \sup\limits_{\partial U} \phi(x) > 0$ and $A := \max\{\frac{M}{m}, 1\}$. We note that
$u_{(j)} \geq u_{\Omega} \geq m \geq \frac{\phi}{A}$ on $\partial U$, and
$u_{(j)} = \infty > \frac{\phi}{A}$ on $U \cap \partial \Omega_{(j)}$.
By Proposition \ref{Prop6-1}, $\frac{\phi}{A}$ is again a positive $k$-admissible subsolution of \eqref{eq4-1}. By Theorem \ref{Theorem6-1}, we arrive at
$u_{(j)} \geq \frac{\phi}{A}$ in $U \cap \Omega_{(j)}$.
Consequently, we obtain $u_{\Omega} \geq \frac{\phi}{A}$ in $U \cap \Omega$.
\end{proof}

In what follows we assume that $u_{\Omega} \not\equiv 0$.
Similar to section 3, let $\rho(x)$ be the distance of $x$ to $\Gamma$, and
\[ \Gamma_{\rho_0}  = \big\{ x \in \Omega \, \big| \, \rho(x) < \rho_0 \big\}. \]
For $\rho_0$ sufficiently small, define on $\Gamma_{\rho_0} \setminus \Gamma$
\[ \phi(x) = c \rho^{1 - \frac{n}{2}} (x), \]
where $c > 0$ is a constant to be determined later.

Denote
\[ v_m = \big( \underbrace{n - m, \ldots, n - m}_{n - m + 1}, \underbrace{2 - m, \ldots, 2 - m}_{m - 1} \big). \]
As $\rho \rightarrow 0$,
\[ \lambda \bigg(  \frac{2}{n - 2} \phi^{- \frac{n + 2}{n - 2}}  W[ \phi ]  \bigg) \rightarrow  c^{- \frac{4}{n - 2}} v_m. \]
If we assume that $v_m \in \Gamma_{k}$, then there exists $c > 0$ sufficiently small such that
\[\begin{aligned}
\sigma_k \Big( c^{- \frac{4}{n - 2}} v_m \Big) + \underline{\alpha} \sigma_{k - 1} \Big( c^{- \frac{4}{n - 2}} v_m \Big) > \overline{\alpha}_0.
\end{aligned} \]
As $\rho_0 > 0$ sufficiently small, within $\{ 0 < \rho < \rho_0 \}$,
\[  \sigma_k \bigg( \frac{2}{n - 2} {\phi}^{- \frac{n + 2}{n - 2}}  W[ \phi ] \bigg) + \underline{\alpha} \sigma_{k - 1} \bigg( \frac{2}{n - 2} {\phi}^{- \frac{n + 2}{n - 2}} W[ \phi ] \bigg) > \overline{\alpha}_0. \]
Therefore, $\Gamma$ is regular by Theorem \ref{Theorem6-3}.

If $v_m \in \mathbb{R}^n \setminus \overline{\Gamma}_k$, as in section 3, we
consider
\[ \psi = \psi_{c, d} = ( c \rho^{- a} + d )^b, \]
where $a$, $b$, $c$ and $d$ are positive constants. Direct calculation shows that
\begin{equation} \label{eq6-5}
 \frac{1}{a b c} \rho^{a + 2} (c \rho^{- a} +  d)^{ \frac{4 b}{n - 2} + 1} \psi^{ - \frac{n + 2}{n - 2}} W [ \psi ]
=  A_m + B_{a b}(\zeta) + \mathcal{O}(\rho),
\end{equation}
where $\zeta$, $A_m$ and $B_{a b}(\zeta)$ are as in \eqref{eq3-1}.
Choose $a$ sufficiently small, and then choose an appropriate $b$ such that $a b$ is slightly larger than $\frac{n}{2} - 1$ such that
\[ \lambda\big( A_m + B_{a b}(\zeta) \big) \notin \overline{\Gamma}_{k} \quad \text{for all } 0 < \zeta < 1. \]
Then choose  $0 < \rho_0 < 1$ further small such that within $0 < \rho < \rho_0$ and for all $0 < \zeta < 1$,
\[ \lambda\big(  A_m + B_{a b}(\zeta) + \mathcal{O}(\rho) \big) \notin  \overline{\Gamma}_{k},  \]
which implies that
\[ \lambda \big( W[\psi] \big) \notin \overline{\Gamma}_{k} \text{  in  } \{ 0 < \rho < \rho_0 \} \text{  for all  } c > 0 \text{  and  }  d > 0. \]
Therefore, we have either $ \lambda \big( W[\psi] \big) \notin \Gamma_{k - 1}$, or, if $ \lambda \big( W[\psi] \big) \in \Gamma_{k - 1}$, then
\[  \sigma_k \bigg( \frac{2}{n - 2} \psi^{- \frac{n + 2}{n - 2}}  W[ \psi ] \bigg) + \alpha \sigma_{k - 1} \bigg( \frac{2}{n - 2} \psi^{- \frac{n + 2}{n - 2}} W[ \psi ] \bigg) \leq 0 < \alpha_0. \]
That is, $\psi$ is a supersolution of \eqref{eq4-1}.
Now, choosing $d$ sufficiently large depending on $\rho_0$ such that on $\rho = \rho_0$ we have $\psi \geq d^{b} \geq u_{\Omega}$. By Proposition \ref{Proposition6-1}, Remark \ref{Remark6-1} and Theorem \ref{Theorem6-1},  the solution of the equation
\[\sigma_k \bigg( \frac{2}{n - 2} u^{- \frac{n + 2}{n - 2}}  W[ u ] \bigg) = \underline{\alpha}_0 \]
in the ball $B_{\rho(x)} (x)$ can serve as an upper bound for $u_{\Omega}$. Thus, we can deduce that
\[ u_{\Omega} \leq  \bigg(\frac{4 (n - 1)}{\tilde{\mu}}\bigg)^{\frac{n - 2}{4}}  \rho^{1 - \frac{n}{2}} \leq c^{b} \rho^{- a b} < \psi \]
within $0 < \rho \leq \delta$, where $\delta > 0$ is a sufficiently small constant depending on $c$. By Theorem \ref{Theorem6-1} we have $u_{\Omega} \leq \psi$ in $\delta < \rho < \rho_0$. Letting $\delta \rightarrow 0$, we arrive at $u_{\Omega} \leq \psi$ in $0 < \rho < \rho_0$. Next letting $c \rightarrow 0$ we deduce that $u_{\Omega} \leq d^{b}$ in $0 < \rho < \rho_0$.
Hence $\Gamma$ is not regular according to Definition \ref{def6-1}.

\end{document}